\documentclass[reqno]{amsart}
\usepackage{amsmath,amsfonts,amssymb,amscd,multicol, mathtools}
\usepackage[all]{xypic}
\usepackage[normalem]{ulem}
\usepackage{verbatim,color}
\usepackage{bbm}
\usepackage{leftidx}
\usepackage{fullpage}

\usepackage[colorlinks]{hyperref}
\usepackage{color}
\definecolor{darkgreen}{RGB}{55,138,0}

\hypersetup{citecolor = darkgreen, linkcolor = darkgreen}

\newcommand{\on}{\operatorname}

\newcommand{\bpr}{\begin{proof}}
\newcommand{\epr}{\end{proof}}

\newcommand{\mc}{\mathcal}

\newcommand{\mb}{\mathbb}

\newcommand{\GKdim}{\operatorname{GKdim}}

\newcommand{\ev}{\operatorname{ev}}
\newcommand{\coev}{\operatorname{coev}}

\newcommand{\catMod}{\operatorname{-Mod}}

\newcommand{\rcatMod}{\operatorname{Mod-}}

 \DeclareMathOperator{\rann}{r.ann}

\newcommand{\kk}{\Bbbk}

\newcommand{\op}{\operatorname{op}}
\newcommand{\cop}{\operatorname{cop}}

\newcommand{\coinv}{\operatorname{coinv}}

\newcommand{\beq}{\begin{equation}}
\newcommand{\eeq}{\end{equation}}

\newcommand{\Hom}{{\rm Hom}}

\newcommand{\End}{{\rm End}}

\newcommand{\Ext}{{\rm Ext}}

\newcommand{\Bimod}{\operatorname{-Bimod}}

\newcommand{\otensor}{\overline{\otimes}}

\newcommand{\lMod}{\operatorname{-Mod}}
\newcommand{\lmod}{\operatorname{-mod}}
\newcommand{\rMod}{\operatorname{Mod-}}

\numberwithin{equation}{section}
 \theoremstyle{plain}
\newtheorem{theorem}[equation]{Theorem}

\newtheorem{lemma}[equation]{Lemma}

\newtheorem{corollary}[equation]{Corollary}
\newtheorem{proposition}[equation]{Proposition}

\theoremstyle{definition}
\newtheorem{question}[equation]{Question}
\newtheorem{definition}[equation]{Definition}

\newtheorem{remark}[equation]{Remark}

\newtheorem{standing-hypothesis}[equation]{Standing Hypothesis}

\begin{document}

\title{Homological integrals for weak Hopf Algebras}
\author{D. Rogalski, R. Won, J.J. Zhang}

\address{Rogalski: Department of Mathematics, University of California, San Diego,
La Jolla, CA 92093, USA}

\email{drogalski@ucsd.edu}

\address{Won:
Department of Mathematics, George Washington University, Washington, DC 20052,
USA}
\email{robertwon@gwu.edu}

\address{Zhang: Department of Mathematics, Box 354350,
University of Washington, Seattle, WA 98195, USA}

\email{zhang@math.washington.edu}

\begin{abstract}
We introduce the notion of a homological integral for an infinite-dimensional weak Hopf algebra
and use the homological integral to prove several structure theorems. For example,
we prove that the Artin--Schelter property and the 
Van den Bergh condition are equivalent for a noetherian 
weak Hopf algebra, and that the antipode is automatically 
invertible in this case.  We also prove a decomposition 
theorem that states that any weak 
Hopf algebra finite over an affine center is a direct 
sum of Artin--Schelter Gorenstein, Cohen--Macaulay, 
GK dimension homogeneous weak Hopf algebras. 
\end{abstract}





\maketitle


\section*{Introduction}
\label{yysec0}

Throughout let $\kk$ be a base field. If $H$ is a 
finite-dimensional Hopf algebra over $\kk$, the theory 
of integrals is of primary importance in understanding 
the structure of $H$. This theory has been extended to 
many more general settings: for example, if $H$ is a 
finite-dimensional \emph{weak} Hopf algebra, integrals 
were defined for $H$ in the earliest papers on the 
subject \cite{BNS}.  In another direction, when $H$ is 
an infinite-dimensional Hopf algebra satisfying the 
Artin--Schelter Gorenstein condition, Lu, Wu, 
and third-named author defined the homological integral 
for $H$, which allows for many of the applications of 
the integral to be extended to this setting. The goal 
of this paper is to study the notion of integral for 
the common generalization of the cases above, when $H$ 
is a infinite-dimensional weak Hopf algebra.

We first review some of the history and important results 
about integrals. Let $H$ be a Hopf algebra over $\kk$, 
with comultiplication $\Delta: H \to H \otimes_{\kk} H$, 
counit $\epsilon: H \to \kk$ and antipode $S: H \to H$.  
Let $_H \kk$ indicate the \emph{trivial left $H$-module}, 
where $h \lambda = \epsilon(h) \lambda$ for 
$h \in H, \lambda \in \kk$. A \emph{left integral} for 
$H$ is an element $\int^{\ell} \in H$ such that 
$h \int^{\ell} = \epsilon(h) \int^{\ell}$ for all $h \in H$.  
In other words, a nonzero integral $\int^{\ell}$ generates 
a $1$-dimensional left ideal $\kk \int^{\ell}$ of $H$ which 
is isomorphic as a left $H$-module to the trivial module 
$_H\kk$.  The space of all left integrals is denoted 
$\int_H^{\ell}$. Right integrals are defined similarly.  

It is well-known that a left or right integral exists in 
$H$ if and only if $\dim_{\kk} H < \infty$. In this 
case there is a unique (up to a scalar) nonzero 
left integral $\int^{\ell}$, so that 
$\int^{\ell}_H = \kk \int^{\ell}$. The space $\int^{\ell}_H$ 
is also a right ideal, so for all $h \in H$, 
$\int^{\ell} h =\rho(h)\int^{\ell}$ for some ring 
homomorphism $\rho: H \to \kk$; that is, $\rho$ is a 
grouplike element of $H^*$. Note that we can also identify 
$\int^{\ell}_H$ with $\Hom_H(\leftidx{_{H}} \kk, H)$. 
Similarly, the space of right integrals $\kk \int^{r}$ is 
one-dimensional and defines a grouplike element $\sigma$ 
by considering it as a left $H$-module. Integrals have many 
uses in the basic structure theory of finite-dimensional 
Hopf algebras.  For example, there is Larson--Sweedler's 
version of Maschke's Theorem for Hopf algebras: $H$ is 
semisimple if and only if $\epsilon(\int^{\ell}_H) \neq 0$.  

Next, consider the case that $H$ is a possibly 
infinite-dimensional Hopf algebra over $\kk$. Recall 
that an algebra $H$ is \emph{Artin--Schelter (AS) 
Gorenstein} if (i) $H$ has finite injective dimension 
$d$ as a left $H$-module; (ii) 
$\Ext^i_H(\leftidx{_H} \kk, H) = 0$ for $i \neq d$, 
while $\dim_{\kk} \Ext^d_H(\leftidx{_H}\kk, H) = 1$; 
and the right sided versions of (i) and (ii) also hold.  
If moreover $H$ has finite global dimension $d$, then 
$H$ is called \emph{Artin--Schelter regular}. Brown 
and Goodearl conjectured that all noetherian Hopf algebras 
are AS Gorenstein. This is known to be true, for example, 
when $H$ is an affine noetherian polynomial identity 
(PI) algebra, by work of Wu and third-named author 
\cite{WZ1}.  For an AS Gorenstein Hopf algebra, there is 
a natural one-dimensional module which plays the same 
role as the integral does in the finite-dimensional case.

\begin{definition} \cite{LWZ07}
\label{yydef0.1}
Let $H$ be an AS Gorenstein Hopf algebra of 
injective dimension $d$. The \emph{space of left 
homological integrals} is the one-dimensional vector 
space
\[
{\textstyle \int}^{\ell}_H = \Ext_H^d(\leftidx{_H}\kk, H),
\]
and any nonzero element in this space is called a 
\emph{left homological integral}.
\end{definition}

Similar to the finite-dimensional case, the space of 
left integrals $\int^{\ell}_H$ is an $(H, H)$-bimodule 
which is isomorphic to the trivial module $_H \kk$ on 
the left, but defines some grouplike element 
$\rho \in H^*$ on the right. Many important theorems 
about finite-dimensional Hopf algebras have been 
generalized to the case of AS Gorenstein Hopf 
algebras $H$ by using the homological integral. 
For example, let $H^{e}=H\otimes_{\kk} H^{\op}$ be 
the enveloping algebra of $H$. Brown and the 
third-named author showed that $\Ext^d_{H^e}(H, H^e) 
\cong \leftidx{^1}H^{\mu}$.  Here $\mu = \xi 
\circ S^2$ is the \emph{Nakayama automorphism} of 
$H$, where $\xi$ is the left winding automorphism 
$\xi(h) =  \sum \rho(h_1)h_2$ associated to the 
grouplike element $\rho \in H^*$ given by $\int^{\ell}_H$.  
The homological integral has also been an essential tool 
in the study of Hopf algebras of low 
Gelfand--Kirillov (GK) dimension; particularly important 
is the \emph{integral order} of $H$, that is, the order 
of the grouplike element $\rho$ in the group of grouplikes 
in $H^*$.  Due to work of a number of authors, the affine 
prime Artin--Schelter regular Hopf algebras with $\GKdim(H) = 1$ have 
been completely classified.  See \cite{BZ21} for a 
survey of this work.

In this paper, we extend the important theory of integrals 
to the setting of infinite-dimensional weak Hopf algebras. 
A weak Hopf algebra $(H, m, u, \Delta, \epsilon, S)$ is a 
structure similar to a Hopf algebra (we review the 
formal definition in Section~\ref{yysec1}). For the 
purposes of the introduction, the most important 
feature of a weak Hopf algebra $H$, which motivates 
this concept, is the corresponding multiring category 
structure on its category of left modules, $(H\lMod,
\otensor^{\ell}, H_t)$. Here, $\otensor^{\ell}$ is a 
natural monoidal product defined using $\Delta$, and 
$H_t$ is a unit object defined as follows. The weak 
Hopf algebra comes along with a \emph{target counital 
map} $\epsilon_t: H \to H$.  Then $H_t = \epsilon_t(H)$ 
with a natural left $H$-module structure given by 
$h \cdot x = \epsilon_t(hx)$ for $h \in H, x \in H_t$.  Similarly there is a \emph{source counital map} 
$\epsilon_s$, a \emph{source counital subalgebra} 
$H_s = \epsilon_s(H)$ which is a right $H$-module, and 
a multiring category $(\rMod H, \otensor^{r}, H_s)$.

If $H$ is a weak Hopf algebra which is finite-dimensional 
over $\kk$, a theory of integrals which closely parallels 
that for finite-dimensional Hopf algebras is known.  
A \emph{left integral} for $H$ is an element 
$\int^{\ell} \in H$ such that $h \int^{\ell} = 
\epsilon_t(h) \int^{\ell}$ for all $h \in H$.  The 
set $\int^{\ell}_H$ of all left integrals is a right 
ideal of $H$, called the \emph{space of left integrals} 
in $H$.  We also have an isomorphism of right 
$H$-modules $\Hom_H(H_t, H) \cong \int^{\ell}_H$, via 
the map sending $f: H_t \to H$ to $f(1)$.  

Now let $H$ be an arbitrary weak Hopf algebra over 
$\kk$. As in the case of infinite-dimensional Hopf 
algebras, we expect to be able to define a reasonable 
integral only when the algebra has good homological 
properties, so we extend the definition of AS Gorenstein 
as follows.  We say that a $\kk$-algebra $A$ is 
\emph{Artin--Schelter (AS) Gorenstein} if (i) 
${} _A A$ has finite injective dimension $d$; (ii) for 
all finite-dimensional left $A$-modules $V$ and 
$i \neq d$ we have $\Ext_A^i(V, A) = 0$, while 
$\Ext_A^d(V, A)$ is finite-dimensional over $\kk$; and 
(iii) the analogous properties hold also on the right. 
When $A=H$ is a (non-weak) Hopf algebra, then this 
is equivalent to the definition of AS Gorenstein given 
before Definition \ref{yydef0.1}.

For an AS Gorenstein weak Hopf algebra, it is not 
hard to guess at a definition of integral that is 
a common generalization of the definition of integral 
for finite-dimensional weak Hopf algebras and the 
homological integral for infinite-dimensional Hopf 
algebras.  

\begin{definition}
\label{yydef0.2}
Let $H$ be an AS Gorenstein weak Hopf algebra of 
injective dimension $d$ with counital subalgebras 
$H_s$ and $H_t$.  The \emph{left homological integral} 
of $H$ is defined to be the right $H$-module 
$\int^{\ell}_H = \Ext^d_H(H_t, H)$.  Similarly, the 
\emph{right homological integral} of $H$ is the left 
$H$-module $\int^{r}_H = \Ext^d_{H^{\op}}(H_s, H)$.
\end{definition}

We first prove the following basic result about the 
integral, which generalizes the fact that the 
homological integral of an AS Gorenstein Hopf 
algebra determines a grouplike element of the dual.

\begin{proposition}[Proposition \ref{yypro3.7}]
\label{yypro0.3}
Let $H$ be a noetherian AS Gorenstein weak Hopf 
algebra. The left integral $\int^{\ell}_H$ 
{\rm{(}}resp. right integral $\int^r_H${\rm{)}} is 
an invertible object in the tensor category of right 
$H$-modules {\rm{(}}resp. left $H$-modules{\rm{)}}.
\end{proposition}

Next, we give generalizations of the work of Brown 
and third-named author in \cite{BZ}. We say that 
$H$ \emph{satisfies the Van den Bergh condition} if 
there is $d \geq 0$ such that $\Ext^i_{H^e}(H, H^e) =0$ 
for $i \neq d$, while $\Ext^d_{H^e}(H, H^e) = U$ is 
an invertible $H$-bimodule (called the 
\emph{Nakayama bimodule}). 

\begin{theorem}
\label{yythm0.4}
Let $H$ be a noetherian weak Hopf algebra.
\begin{enumerate}
\item[(1)]
{\rm{[Proposition \ref{yypro4.5}]}}
For all $i \geq 0$, $\Ext^i_{H^e}(H, H^e) \cong 
\Ext_H^i(H_t, H) \otensor^r H^{S^2}$ as 
$(H, H)$-bimodules, where the right $H$-module 
structure comes from the monoidal product 
$\otensor^r$ in $\rMod H$, and the left $H$-module 
structure comes from the left side of $H^{S^2}$.
\item[(2)]
{\rm{[Theorem \ref{yythm4.7}]}}
$H$ satisfies the Van den Bergh condition if and 
only if $H$ is AS Gorenstein. When this holds, then 
$U:= \Ext^d_{H^e}(H, H^e)$ is an invertible 
$H$-bimodule, where $d$ is the injective 
dimension of $H$.  In particular, 
$U \cong \int^{\ell}_H \otensor^r H^{S^2}$ as 
$(H, H)$-bimodules.
\end{enumerate}
\end{theorem}

While exotic examples of Hopf algebras exist for 
which the antipode $S$ is not invertible, it is 
natural to ask whether $S$ must be a bijection for 
reasonably well-behaved Hopf algebras.  Skryabin 
has conjectured that this is the case for any 
noetherian Hopf algebra. Using techniques from 
homological integrals, L\"u, Oh, Wang, and Yu 
proved that any noetherian AS Gorenstein Hopf 
algebra has a bijective antipode $S$ 
\cite[Corollary 0.4]{LOWY}.  Moreover, Brown and 
Goodearl have conjectured that any noetherian Hopf 
algebra is automatically AS Gorenstein, which 
would imply Skryabin's conjecture.  The 
Brown--Goodearl conjecture has been proved for 
PI Hopf algebras \cite{WZ1}.  

Using our results on the homological integrals of 
weak Hopf algebras, we are able to extend the 
results  from \cite{LOWY} to this case:

\begin{theorem}[Theorem~\ref{yythm4.6}]
\label{yythm0.5}
Let $H$ be a noetherian weak Hopf algebra which 
is a finite sum of AS Gorenstein algebras.  
Then the antipode $S$ is a bijection.
\end{theorem}

To explain the hypothesis of the preceding theorem, 
note that in \cite{RWZ1}, we proved that if $H$ 
is a weak Hopf algebra that is a finite module 
over its affine center, then $H$ is a finite 
direct sum \emph{as algebras} of noetherian AS 
Gorenstein algebras \cite[Theorem 0.3]{RWZ1}.  
While this is enough to define a version of the 
homological integral (see Section~\ref{yysec3}) 
and thus to prove Theorem~\ref{yythm0.5}, the 
strongest analog of the Brown--Goodearl conjecture 
for weak Hopf algebras would state that a 
noetherian weak Hopf algebra must be a finite 
direct sum of AS Gorenstein weak Hopf algebras. 
Direct sums are unavoidable, since the direct sum 
of two weak Hopf algebras is again a weak Hopf 
algebra, but if the two algebras have different 
injective dimensions then the direct sum cannot 
be AS Gorenstein.  In Section~\ref{yysec5}, we 
improve our result from \cite{RWZ1} to obtain 
this stronger form of the Brown--Goodearl 
conjecture, again when $H$ is a finite module 
over an affine center (see Theorem~\ref{yythm5.4}).

The following theorem summarizes what we now know 
about the structure of weak Hopf algebras finite 
over affine centers.

\begin{theorem}
\label{yythm0.6}
Let $H$ a weak Hopf algebra that is finitely 
generated as a module over its affine center.
\begin{enumerate}
\item[(1)]
The antipode of $H$ is bijective.
\item[(2)]{\rm{[Decomposition Theorem]}}
$H$ is a finite direct sum {\rm{(}}as weak Hopf 
algebras{\rm{)}} of AS Gorenstein weak Hopf algebras.
\item[(3)]
Every AS Gorenstein weak Hopf subalgebra summand in 
part {\rm{(2)}} satisfies the Van den Bergh condition.
\end{enumerate}
\end{theorem}

Note that part (2) of the above corollary answers 
\cite[Question 8.2]{RWZ1} and part (1) answers 
\cite[Question 8.3]{RWZ1} in the case when $H$ is 
finite over its affine center.

Just as the theory of homological integrals has 
led to a classification of affine regular prime 
Hopf algebras of GK dimension $1$, we hope to use 
the results of this paper to study regular weak 
Hopf algebras of GK dimension 1 in future work.

A weak bialgebra is a special case of a more general 
construction called a \emph{bialgebroid}, 
which is a kind of bialgebra over a general base algebra 
$R$ (weak bialgebras are the case where $R$ is finite-dimensional semisimple).  
Kowalzig and Kr\"ahmer have studied a 
version of Poincar\'e duality for 
bialgebroids over $R$ with a kind of antipode.  
In particular, when specialized to weak Hopf algebras, \cite[Theorem 1]{KoKr} gives an isomorphism between Ext and Tor that is of a similar flavor as some of our results in Sections~\ref{yysec3} 
and \ref{yysec4} below.  there may well be generalizations 
of our results to the setting of bialgebroids with antipode, 
but we do not pursue this here.


\section{Preliminaries on weak Hopf algebras}
\label{yysec1}

Throughout we fix a field $\kk$ and all objects 
will be vector spaces over $\kk$.  The term 
\emph{finite-dimensional} will refer to the 
dimension of an object over $\kk$ unless otherwise 
specified. In this section we review the 
definition of a weak Hopf algebra and some of the 
basic results we will need. The reader can find 
more details in the survey article \cite{NV}.

A \emph{weak bialgebra} over $\kk$ is a 
$\kk$-vector space $H$ with both a $\kk$-algebra 
$(H,m,u)$ and a $\kk$-coalgebra structure 
$(H, \Delta, \epsilon)$ satisfying some 
compatibility axioms.  We frequently use sumless 
Sweedler notation to indicate the result of 
applying $\Delta$ to an element; so $\Delta(h) 
= h_1 \otimes h_2$. As with a usual bialgebra, 
we require $\Delta(gh) = \Delta(g)\Delta(h)$ for 
all $g, h \in H$. However, we do not require 
$\Delta(1) = 1$ or $\epsilon(gh) = 
\epsilon(g)\epsilon(h)$ in general; instead we 
specify
\[
(\Delta \otimes \on{id}) \circ \Delta(1) 
= (\Delta(1) \otimes 1)(1 \otimes \Delta(1)) 
= (1 \otimes \Delta(1))(\Delta(1) \otimes 1)
\]
and 
\[
\epsilon(fgh) = \epsilon(fg_1) \epsilon(g_2h) 
= \epsilon(fg_2)\epsilon(g_1h)\ \text{for all}\ 
f, g, h \in H.
\]
We think of the first of these equations as a 
kind of weak unitality of $\Delta$, and the second 
as a kind of weak multiplicativity of $\epsilon$.  

Applying the Sweedler notation we have 
$\Delta(1) = 1_1 \otimes 1_2$, which plays an 
important role in many formulas.  There are two 
important variants of $\epsilon$ that play a major 
role: defining 
\[
\epsilon_s(h) = 1_1 \epsilon(h 1_2)\ \qquad 
\text{and}\ \qquad \epsilon_t(h) = \epsilon(1_1h)1_2
\]
then $\epsilon_s, \epsilon_t: H \to H$ are called 
the \emph{counital maps} of $H$. The maps 
$\epsilon_s, \epsilon_t$ are not ring homomorphisms 
in general, but they are idempotent. The images of 
these maps are denoted $H_s = \epsilon_s(H)$ and 
$H_t = \epsilon_t(H)$ and are called the 
\emph{source} and \emph{target} \emph{counital 
subalgebras}, respectively. The subspaces $H_s$ and 
$H_t$ are finite-dimensional separable subalgebras 
of $H$ which commute with each other. They also 
have the following alternate characterization:
\[
H_s = \{h \in H \mid \Delta(h) = 1_1 \otimes h1_2\}\ 
\qquad \text{and}\ \qquad 
H_t = \{h \in H \mid \Delta(h) = 1_1h \otimes 1_2\}.
\]

A weak bialgebra is called a \emph{weak Hopf algebra} 
if there exists a $\kk$-linear \emph{antipode} 
$S: H \to H$ such that for all $h \in H$,
\begin{enumerate}
\item[(1)]
$h_1 S(h_2)  = \epsilon_t(h)$;
\item[(2)] 
$S(h_1) h_2 = \epsilon_s(h)$; and 
\item[(3)] 
$S(h_1)h_2S(h_3) = S(h)$.
\end{enumerate}
It follows from the axioms that $S$ is an anti-algebra 
and anti-coalgebra homomorphism of $H$.  Moreover, 
$S \circ \epsilon_s = \epsilon_t \circ S$ and 
$S \circ \epsilon_t = \epsilon_s \circ S$.  In 
particular, 
$S(H_s) = H_t$ and $S(H_t) = H_s$.
We will not assume in this paper that $S$ is 
bijective, since one of our goals is to show that 
this is automatic for certain nice weak Hopf 
algebras. 

One of the main reasons for considering weak Hopf 
algebras is that they lead to interesting 
monoidal categories. We refer the reader to 
\cite{EGNO} for the basic definitions and theory 
of monoidal categories.  We generally refer to 
a monoidal category as a triple 
$(\mc{C}, \otimes, \mathbbm{1})$, where $\mc{C}$ 
is an abelian category, $\otimes$ a monoidal 
product, and $\mathbbm{1}$ the unit object. All 
of the monoidal categories we consider will have 
standard canonical choices of associativity and 
unit isomorphisms and so we omit them from the 
notation.

For an algebra $A$, we write $A \lMod$ for the 
abelian category of left $A$-modules and 
$\rMod A$ for the category of right $A$-modules. 
If $B$ is another algebra, the category of $(A, B)$-bimodules 
will be denoted $(A, B) \Bimod$.  When $A = B$ we also 
refer to an $(A, A)$-bimodule as an \emph{$A$-bimodule}.

If $H$ is a weak bialgebra, for $M, N \in H 
\lMod$, we define 
\[
M \otensor^{\ell} N = \Delta(1)(M \otimes_{\kk} N) 
= \left\{ \sum m_i \otimes n_i \in M \otimes_{\kk} 
N \mid \sum m_i \otimes n_i = \sum 1_1 m_i \otimes 
1_2 n_i \right\},
\]
using that $\Delta(1)$ is an idempotent in 
$H \otimes H$. Then $M \otensor^{\ell} N 
\in H \lMod$ has a left $H$-module structure 
defined by $h \cdot (m \otimes n) = h_1m 
\otimes h_2n$.  The target counital subalgebra 
$H_t$ is a left $H$-module via the action 
$h \cdot x = \epsilon_t(hx)$ for $h \in H, 
x \in H_t$. The category $H \lMod$ is a 
monoidal category with product 
$\otensor^{\ell}$ and unit object $H_t$, 
where the associativity and unit constraints 
are just induced by the canonical 
associativity and unit constraints of 
$\otimes_{\kk}$.  By definition, it is clear 
that $\otensor^\ell$ is bilinear on morphisms 
and exact in each tensor coordinate.  
Symmetrically, $\rMod H$ is a monoidal category 
with analogous properties, where the monoidal 
product is $M \otensor^r N = (M \otimes_{\kk} N)
\Delta(1)$ and the unit object is $H_s$, which 
is a right $H$-module via 
$x \cdot h = \epsilon_s(xh)$.

If $H$ is a weak \emph{Hopf} algebra,  then any 
finite-dimensional module $M \in H \lMod$ has a 
\emph{left dual} $M^*$ in the sense of 
\cite[Chapter 2]{EGNO}, where $M^* = 
\Hom_\kk(M, \kk)$ is the $\kk$-linear dual with 
action $[h \cdot \phi](m) = \phi(S(h)m)$, for 
$\phi \in M^*$, $h \in H$, $m \in M$.  If $S$ is 
bijective, then any such $M$ also has a right 
dual $^* M = \Hom_\kk(M, \kk)$ with 
$[h \cdot \phi](m) = \phi(S^{-1}(h)m)$.  Since 
we do not assume that $S$ is bijective in this 
paper, we will work primarily with left duals. 
Similar comments of course apply to $\rMod H$; 
finite-dimensional objects in this category have 
left duals defined by an analogous formula, and 
also right duals if $S$ is bijective.

There is another important way of thinking about 
the monoidal product in $H \lMod$. For convenience 
we describe this only when $H$ is a weak Hopf 
algebra, though in fact there is a way to define 
it for any weak bialgebra.  Although the antipode 
$S$ need not be bijective in general, it is known 
that $S: H_s \to H_t$ and $S: H_t \to H_s$ are 
bijections, so we can write $S^{-1}(x)$ when $x$ 
is an element of $H_s$ or $H_t$.  Now for any 
$M \in H \lMod$, we can define an ``underlying" 
$H_t$-bimodule structure on $M$, where the 
left action is the restriction of the left 
$H$-action on $M$ to $H_t$, and the right action 
is defined by $m * x = S^{-1}(x) m$ for 
$m \in M, x \in H_t$. Now one may check that 
for $M, N \in H \lMod$, there is a natural 
identification $M \otensor^{\ell} N = 
M \otimes_{H_t} N$ as $H_t$-bimodules 
with the trivial formula $(m \otimes n) 
\mapsto (m \otimes n)$, where the left 
$H$-module structure on $M \otimes_{H_t} N$ is 
still given by the formula $h \dot (m \otimes n) 
= h_1m \otimes h_2n$. As such this gives a 
monoidal functor from 
$(H \lMod, \otensor^{\ell}, H_t)$ to 
$((H_t, H_t) \Bimod, \otimes_{H_t}, H_t)$.  
Note that we will also suppress the choice 
of isomorphisms that is part of the definition 
of monoidal functor, as they will always be 
the obvious canonical ones. As usual, the 
monoidal category $\rMod H$ of right 
$H$-modules can be described similarly; any 
$M \in \rMod H$ has an underlying 
$H_s$-bimodule, and $M \otensor^r N 
= M \otimes_{H_s} N$.

For a right $H$-module $M$, let $^S M$ be the 
left $H$-module which is $M$ as a vector space, 
with left action $h \cdot m = mS(h)$.  Similarly, 
for a left module $N$, let $N^S$ is the induced 
right module with action $n \cdot h = S(h)n$.   
Note that if $M$ is an $(H,B)$-bimodule for some 
other $\kk$-algebra $B$, then $M^S$ is a right 
$H \otimes_{\kk} B$-module; that is, the 
$B$-module structure is maintained on the same 
side.  

For any monoidal category $(\mc{C}, 
\otimes, \mathbbm{1})$, the ``opposite product" 
$\otimes^{\op}$, where $M \otimes^{\op} N = 
N \otimes M$, also gives a monoidal structure 
to $(\mc{C}, \otimes^{\op}, \mathbbm{1})$. The 
following property of the operation $S$ follows 
easily from the fact that $S$ is an 
anti-homomorphism of coalgebras.

\begin{lemma} 
\label{yylem1.1}
Let $H$ be a weak Hopf algebra.  Then $( -)^S: 
H \lMod \to \rMod H$ gives a monoidal functor 
$(H \lMod, \otensor^{\ell}, H_t) \to 
(\rMod H, (\otensor^r)^{\on{op}}, H_s)$. In 
particular, $(H_t)^S \cong H_s$ and 
$(M \otensor^{\ell} N)^S \cong N^S 
\otensor^r M^S$ as right $H$-modules, for 
$M, N \in H \catMod$.
\end{lemma}

We will frequently use throughout the paper 
that additional bimodule structures are 
maintained by the monoidal products.

\begin{lemma} 
\label{yylem1.2}
Let $H$ be a weak Hopf algebra. Suppose that 
$M \in (H, A) \Bimod$ and $N \in (H, B) 
\Bimod$ for some algebras $A$ and $B$. 
Then $M \otensor^{\ell} N \in 
(H, A \otimes_{\kk} B) \Bimod$. A similar 
result holds for $\otensor^r$.
\end{lemma}

\begin{proof}
This is immediate from the fact that the 
monoidal product $\otensor^{\ell}$ on 
$H \lMod$ is bifunctorial and bilinear. 
Thus, for example, given $a \in A$ the 
right multiplication map $r_a: M \to M$ 
induces a right multiplication map 
$r_a \otimes 1: M \otensor^{\ell} N 
\to M \otensor^{\ell} N$ and this makes 
$M \otensor^{\ell} N$ into a right 
$A$-module. 
\end{proof}

Let $W\in H\lMod$. Then the functor 
$W\otensor^{\ell} -$ is an exact functor 
$H\lMod \to H\lMod$. By the Eilenberg--Watts 
theorem, there is an $H$-bimodule, 
denoted by ${\mathbb F}^L(W)$, such that
$W\otensor^{\ell} -$ is naturally 
isomorphic to ${\mathbb F}^L(W)\otimes_H -$. 
Similarly the functor $-\otensor^{\ell} W$ 
is an exact functor $H\lMod \to H\lMod$ and 
there is an $H$-bimodule  
${\mathbb F}^R(W)$ such that 
$-\otensor^{\ell} W$ is naturally isomorphic 
to ${\mathbb F}^R(W)\otimes_H -$. The 
following lemma follows easily from the 
Eilenberg--Watts theorem and 
Lemma~\ref{yylem1.2}.

\begin{lemma}
\label{yylem1.3} 
Retain the above notation.
\begin{enumerate}
\item[(1)]
Both ${\mathbb F}^L(W)$ and ${\mathbb F}^R(W)$ 
are $H$-bimodules that are flat on the right. 
\item[(2)]
${\mathbb F}^L(W)= W\otensor^{\ell} H$ where 
the right $H$-module structure on 
$W\otensor^{\ell} H$ is determined by the 
right $H$-action on the second tensorand.
Consequently, ${\mathbb F}^L: W\to 
{\mathbb F}^L(W)$ is a functor from 
$H\lMod$ to $(H,H)\Bimod$.
\item[(3)]
${\mathbb F}^R(W)= H\otensor^{\ell} W$ where 
the right $H$-module structure on 
$H\otensor^{\ell} W$ is determined by the 
right $H$-action on the first tensorand. 
Consequently, ${\mathbb F}^R$ is a functor 
from $H\lMod$ to $(H,H)\Bimod$.
\item[(4)]
For $W_1,W_2\in H\lMod$, we have canonical 
$H$-bimodule isomorphisms 
$\mathbb{F}^L(H_t) \cong H$, 
$\mathbb{F}^R(H_t) \cong H$,
$${\mathbb F}^L(W_1\otensor^{\ell} W_2)
\cong {\mathbb F}^L(W_1)\otimes_H 
{\mathbb F}^L(W_2), \qquad {\text{and}} 
\qquad {\mathbb F}^R(W_1\otensor^{\ell} W_2)
\cong {\mathbb F}^R(W_2)\otimes_H 
{\mathbb F}^R(W_1),$$
which give ${\mathbb F}^L$ and 
${\mathbb F}^R$ the structure of 
monoidal functors from $H\lMod$ to 
$(H,H)\Bimod$.
\end{enumerate}
\end{lemma}

Similar results hold in the category 
$\rMod H$ of right $H$-modules. If 
$V\in \rMod H$, we let $\mathbb{G}^L(V)$ 
be the $H$-bimodule such that 
$V\otensor^{r} -$ is naturally isomorphic 
to $-\otimes_H {\mathbb G}^L(W)$, and 
let $\mathbb{G}^R(V)$ be the $H$-bimodule 
such that $-\otensor^{r} V$ is naturally 
isomorphic to $- \otimes_H {\mathbb G}^R(V)$.

\begin{lemma}
\label{yylem1.4} Retain the above notation.
\begin{enumerate}
\item[(1)]
Both ${\mathbb G}^L(V)$ and 
${\mathbb G}^R(V)$ are $H$-bimodules that 
are flat on the left. 
\item[(2)]
${\mathbb G}^L(V)= V\otensor^{r} H$ where 
the left $H$-module structure on 
$V\otensor^{r} H$ is determined by the left 
$H$-action on the second tensorand. 
Consequently, ${\mathbb G}^L: V\to 
{\mathbb G}^L(V)$ is a functor from 
$\rMod H$ to $(H,H)\Bimod$.
\item[(3)]
${\mathbb G}^R(V)= H\otensor^{r} V$ where 
the left $H$-module structure on 
$H\otensor^{r} V$ is determined by the left 
$H$-action on the first tensorand.
Consequently, ${\mathbb G}^R$ is a 
functor from $\rMod H$ to $(H,H)\Bimod$.
\item[(4)]
For $V_1,V_2\in \rMod H$, we have canonical 
$H$-bimodule isomorphisms 
$\mathbb{G}^L(H_s) \cong H$, 
$\mathbb{G}^R(H_s) \cong H$,
$${\mathbb G}^L(V_1\otensor^{r} V_2)
\cong {\mathbb G}^L(V_2)\otimes_H 
{\mathbb G}^L(V_1),
\qquad {\text{and}} \qquad 
{\mathbb G}^R(V_1\otensor^{r} V_2)
\cong {\mathbb G}^R(V_1)\otimes_H 
{\mathbb G}^R(V_2),$$
\end{enumerate}
which give ${\mathbb G}^L$ and 
${\mathbb G}^R$ the structure of monoidal 
functors from $\rcatMod H$ to 
$(H,H)\Bimod$.
\end{lemma}

\begin{remark}
\label{yyrem1.5}
The results on the functors $\mathbb{F}$ 
and $\mathbb{G}$ above have easy extensions 
to bimodules.  For instance, suppose that 
$W \in (H, B) \Bimod$ and $M \in (H, C) 
\Bimod$ for $\kk$-algebras $B$ and $C$. 
We have an isomorphism of left $H$-modules 
$W \otensor^{\ell} M \cong \mathbb{F}^L(W) 
\otimes_{H} M$, for the $(H, H)$-bimodule 
$\mathbb{F}^L(W) \cong W \otensor^{\ell} H$, 
as in Lemma~\ref{yylem1.3}(2). The 
isomorphism is easily seen to hold at the 
level of $(H, B \otimes_{\kk} C)$-bimodules.  
Here, the right $B$ and $C$ structures on 
$W \otensor^{\ell} M$ come from 
Lemma~\ref{yylem1.2}.  By the same lemma, 
$\mathbb{F}^L(W) = W \otensor^{\ell} H$ 
has a right $(B \otimes_{\kk} H)$-module 
structure, and thus $\mathbb{F}^L(W) 
\otimes_{H} M$ maintains a right 
$B$-module structure, as well as a right 
$C$-module structure from $M$.  In this 
way, the isomorphism of functors 
$W \otensor^{\ell} - \cong \mathbb{F}^L(W) 
\otimes_{H} -$ holds as functors from 
$(H, C) \Bimod$ to 
$(H, B \otimes_{\kk} C) \Bimod$. 
Similarly, $\mathbb{F}^L(-)$ can be 
considered as a functor $(H, B) \Bimod 
\to (H, H \otimes_{\kk} B) \Bimod$.

Analogous bimodule extensions hold for 
the functors $\mathbb{F}^R, \mathbb{G}^L, 
\mathbb{G}^R$.
\end{remark}

In fact, the functors $\mathbb{F}$ and 
$\mathbb{G}$ are closely related.  This 
is one consequence of the fundamental 
theorem of Hopf modules for weak Hopf 
algebras, which we review next.

\begin{definition}
\label{yydef1.6}
Let $H$ be a weak Hopf algebra over $\kk$. 
A \emph{left-left Hopf module} is a 
$\kk$-space $M$ which is a left $H$-module 
via $\mu_M: H \otimes_{\kk} M \to M$ 
(where we write $hm := \mu_M(h \otimes m)$) 
and a left $H$-comodule via $\rho_M: M 
\to H \otimes_{\kk} M$ (where we write 
$\rho(m) =  m_{[-1]} \otimes m_{[0]}$), 
such that $ (hm)_{[-1]} \otimes (hm)_{[0]} 
=  h_1m_{[-1]} \otimes h_2 m_{[0]}$ 
for all $h \in H$ and $m \in M$.

When the sides of the actions and coactions 
are clear from context, we refer to a 
left-left Hopf module as simply a 
\emph{Hopf module}.
\end{definition}

Note that if $M$ is a Hopf module, then 
because of the final condition we have 
$ 1_1 m_{[-1]} \otimes 1_2 m_{[0]} = 
m_{[-1]} \otimes m_{[0]}$, and so the 
image of $\rho: M \to H \otimes_{\kk} M$ 
must land in $H \otensor^{\ell} M$.

\begin{theorem}[Fundamental theorem of 
Hopf modules]
\label{yythm1.7}
Let $H$ be weak Hopf algebra with 
counital subalgebras $H_t$ and $H_s$. 
If $M$ is a left-left Hopf module, then 
$M \cong H \otimes_{H_s} M^{\coinv}$ as 
Hopf modules, where
\[
M^{\coinv} = \left\{ m \in M : m_{[-1]} 
\otimes m_{[0]} 
= 1_1 \otimes 1_2 m \right\}
\]
is the subspace of coinvariants of $M$, 
and where $H \otimes_{H_s} M^{\coinv}$ is 
a left $H$-module and comodule via the 
usual structures on the left tensorand 
$H$.  The inverse isomorphisms are given 
by $f: H \otimes_{H_s} M^{\coinv} \to M$ 
with $f(h \otimes m) = hm$, and 
$g: M \to H \otimes_{H_s} M^{\coinv}$ 
with $g(m) = m_{[-2]} \otimes 
S(m_{[-1]})m_{[0]}$.
\end{theorem}

\begin{proof}
This is proved for right-right Hopf 
modules over finite-dimensional weak Hopf 
algebras in \cite[Theorem 3.9]{BNS}. The 
translation of the statement and proof to 
the left-left case is straightforward, and 
the proof works without change in the 
infinite-dimensional case.
\end{proof}

The following is one of the most basic 
examples of a Hopf module and the use of 
the fundamental theorem.

\begin{lemma}
\label{yylem1.8}
Let $H$ be a weak Hopf algebra and let 
$W \in H \lMod$.  
\begin{enumerate}
\item[(1)] 
The left $H$-module $H \otensor^{\ell} W$, 
is a left-left Hopf module via the comodule 
structure $\rho(h \otensor^{\ell} w) 
= h_1 \otimes (h_2 \otensor^{\ell} w)$.  
In particular, $H \otensor^{\ell} W 
\cong H \otimes_{H_s} W$ as left modules.
\item[(2)] 
$H \otensor^{\ell} W \cong H \otensor^r W^S$ 
as $H$-bimodules.
\item[(3)] 
For every $W \in H \catMod$ there is an 
isomorphism ${\mathbb F}^R(W)\cong 
{\mathbb G}^R(W^S)$ as $H$-bimodules. 
These isomorphisms are natural in $W$ 
and therefore ${\mathbb F}^R(-)\cong 
{\mathbb G}^R((-)^S)$ as functors 
$H \catMod \to (H, H) \Bimod$.
\end{enumerate} 
\end{lemma}

\begin{proof} 
(1) As noted in the discussion after 
Definition~\ref{yydef1.6}, the structure 
map $\rho: H \otensor^{\ell} M \to 
H \otimes_{\kk} H \otensor^{\ell} M$ 
defining a left-left Hopf module 
structure on $H \otensor^{\ell} M$  can 
be assumed to land in 
$H \otensor^{\ell} (H \otensor^{\ell} M)$.
Then as in the statement we can simply 
take $\rho = \Delta \otensor^{\ell} 1$, 
and it is easy to check that this does 
define a left-left Hopf module.

By Theorem~\ref{yythm1.7}, there is an 
isomorphism $\psi: H \otensor^{\ell} W 
\cong H \otimes_{H_s} 
(H \otensor^{\ell} W)^{\coinv}$ as Hopf 
modules, where the left $H$-action of 
$H \otimes_{H_s} 
(H \otensor^{\ell} W)^{\coinv}$ is just 
by left multiplication.  

Recall that $H \otensor^{\ell} W = 
H \otimes_{H_t} W$.  Since $H_t$ is 
semisimple, $H$ is faithfully flat as a 
right $H_t$-module. Consequently, the 
map $i: W \to H \otimes_{H_t} W$ given 
by $i(w) = 1 \otimes w$ is injective.  We 
claim that 
$C:=(H \otensor^{\ell} W)^{\coinv} = i(W)$. 
First we show that 
$C=H^{\coinv} \otimes_{H_t} W$. By 
definition, $\Delta(1) (H \otimes_{H_t} W) 
= H \otimes_{H_t} W$. So if 
$\sum h_i \otimes w_i \in C$, where we can 
take $w_i$ to be left independent over 
$H_t$, that is $W = \bigoplus H_t w_i$, 
then 
\begin{gather*}
\rho(h_i \otimes w_i) = h_{i1} \otimes 
h_{i2} \otimes w_i 
= \Delta(1)(1 \otimes (h_i \otimes w_i)) 
= 1_1 \otimes 1_2(h_i \otimes w_i) 
= 1_1 \otimes 1_2 h_i \otimes 1_3 w_i \\
= 1_1  \otimes 1_2 1'_1 h_i \otimes 1'_2 w_i 
= 1_1 \otimes 1_2 h_i \otimes w_i 
= \Delta(1)(1 \otimes h_i) \otimes w_i,
\end{gather*}
where we have used $\Delta^2(1) = 
(\Delta(1) \otimes 1)
(1 \otimes \Delta(1))$. It follows that 
$C \subseteq H^{\coinv} \otimes_{H_t} W$, 
and the other inclusion 
$H^{\coinv} \otimes_{H_t} W \subseteq C$ 
follows in the same way. In a slightly 
different notation, we have
$(H \otensor^{\ell} W)^{\coinv} 
= H^{\coinv} \otensor^{\ell} W$. 
Second we show that $H^{\coinv} 
\otensor^{\ell} W=i(W)$. Note
that $H^{\coinv} = H_s$. If $y \in H_s$ 
then for $y \otimes w \in H^{\coinv} 
\otensor^{\ell} W$ we have 
\[
y \otimes w = (y \cdot 1) \otimes w 
= (1 \cdot S(y)) \otimes w 
= 1 \otimes S(y) w \in i(W),
\]
using that the right $H_t$-action on $H$ 
is given by the left 
$S(H_t) = H_s$-action on $H$. Conversely, 
any $1 \otimes w \in 
H^{\coinv} \otensor^{\ell} W$ so the claim 
that $C = i(W)$ is proved.

Identifying $i(W) = W$, we conclude that 
$H \otensor^{\ell} W \cong H 
\otimes_{H_s} W$ as left $H$-modules, where 
the left $H$-action on the $H \otimes_{H_s} W$ 
is just left multiplication on the left 
tensorand.  

(2) We need to show how the right $H$-module 
structure of $H \otensor^{\ell} W$, given by 
right multiplication on the first tensorand, 
transfers to $H \otimes_{H_s} W$ under the 
isomorphism of (1).

Recall that the isomorphism $f: H \otimes_{H_s} 
(H \otensor^{\ell} W)^{\coinv} \to 
H \otensor^{\ell} W$ given by the fundamental 
theorem is simply $f(g \otimes k \otimes w) 
= g \cdot (k \otimes w) = g_1k \otimes g_2 w$. 
Moreover, the proof of (1) showed that we 
have an isomorphism $W \to 
(H \otensor^{\ell} W)^{\coinv}$ given 
by $w \mapsto (1 \otimes w)$.  Altogether 
this shows that the left $H$-module 
isomorphism $\phi: H \otimes_{H_s} W 
\cong H \otensor^{\ell} W$ of (1) is given 
by the formula $\phi(g \otimes w) 
= g_1 \otimes g_2 w$.

We claim that in $H \otimes_{H_s} W$ the 
right $H$-module structure is given by 
$(g \otimes w) * h = g h_1 \otimes 
S(h_2) w$.  It is enough to show that 
assuming this formula $\phi$ becomes a 
right $H$-module homomorphism.  Using the 
fact that for any $h\in H$ one has 
$h_1 \otimes \epsilon_t(h_2) = 1_1 h 
\otimes 1_2$ 
\cite[Proposition 2.2.1(ii)]{NV}, we calculate
\begin{gather*}
\phi((g \otimes w) * h) 
= \phi(gh_1 \otimes S(h_2) w) 
= g_1h_1 \otimes g_2 h_2 S(h_3) w 
= g_1 h_1 \otimes g_2 
\epsilon_t(h_2)w  \\
= g_11_1 h \otimes g_21_2 w 
= g_1 h \otimes g_2 w 
= (g_1 \otimes g_2 w) \cdot h 
= \phi(g \otimes w) \cdot h.
\end{gather*}
Now identifying $H \otimes_{H_s} W = 
H \otensor^r W^S$, the map $\phi$ is an 
$H$-bimodule isomorphism 
$H \otensor^{\ell} W \to H \otensor^r W^S$.

(3) We have $H$-bimodule isomorphisms 
\[
{\mathbb F}^R(W)\cong  H \otensor^{\ell} W 
\cong  H \otensor^r W^S \cong 
{\mathbb G}^R(W^S)
\]
coming from Lemma~\ref{yylem1.3}(3), part 
(2), and Lemma~\ref{yylem1.4}(3).  It is 
easy to see that the isomorphism of part 
(2) is natural in $W$, so these isomorphisms 
define an isomorphism of functors 
$\mathbb{F}^R(-) \cong \mathbb{G}^R((-)^S)$.
\end{proof}

\begin{remark}
\label{yyrem1.9}
Suppose that $W$ is a $(H, B)$-bimodule for 
some other algebra $B$.  As noted in 
Remark~\ref{yyrem1.5}, $\mb{F}^R(W)$ will 
then be an $(H, H \otimes_{\kk} B)$-bimodule.  
Similarly, $\mb{G}^R(W^S) \cong 
H \otensor^r W^S$ is an 
$(H, H \otimes_{\kk} B)$-module.  
It is easy to check that the isomorphism 
${\mathbb F}^R(W) \cong {\mathbb G}^R(W^S)$ 
holds as $(H, H \otimes B)$-bimodules, and 
that the isomorphism of functors 
${\mathbb F}^R(-) \cong {\mathbb G}^R((-)^S)$ 
holds as functors $(H, B) \Bimod \to 
(H, H \otimes_{\kk} B) \Bimod$.  
\end{remark}

\section{Invertible objects}
\label{yysec2}

Let $(\mc{C}, \otimes, \mathbbm{1})$ be 
a monoidal category. We call an object 
$X$ in $\mc{C}$ \emph{invertible} if it 
has a left dual $X^*$ in the category in 
the sense of \cite[Section 2.10]{EGNO}, 
and where the associated maps 
$\ev_X: X^* \otimes X \to \mathbbm{1}$ 
and $\coev_X: \mathbbm{1} \to X \otimes 
X^*$ are isomorphisms. If $X$ is 
invertible then it also has a right dual 
$^* X$ (so $X$ is \emph{rigid}) and 
moreover $X^* \cong \leftidx{^*} X$. 
Equivalently one may define $X$ to be 
invertible if there is an object $Y$ such 
that $X \otimes Y \cong \mathbbm{1}$ and 
$Y \otimes X \cong \mathbbm{1}$; in this 
case one may choose evaluation and 
coevaluation maps for which $Y$ is a 
left dual of $X$.

We begin with some easy generalities 
about invertible objects.

\begin{lemma}
\label{yylem2.1}
Let $(\mc{C}, \otimes, \mathbbm{1})$ and 
$(\mc{C}', \otimes', \mathbbm{1}')$ be 
abelian monoidal categories and suppose 
that $F: \mc{C} \to \mc{C}'$ is a monoidal 
functor (where we suppress the 
isomorphisms $J_{XY}: F(X \otimes Y) \to 
F(X) \otimes F(Y)$ and $F(\mathbbm{1}) 
\to \mathbbm{1}'$).
\begin{enumerate}
\item[(1)] 
If $X \in \mc{C}$ is invertible then 
$F(X)$ is invertible in $\mc{C}'$.
\item[(2)] 
Suppose that $F$ is exact and faithful 
and that $X^*$ exists in $\mc{C}$. If 
$F(X)$ is invertible in $\mc{C}'$ then 
$X$ is invertible in $\mc{C}$.
\end{enumerate}
\end{lemma}

\begin{proof}
(2) It is standard that $F(X)$ has a 
left dual $F(X)^*$ in $\mc{C}'$ which 
can be identified with $F(X^*)$, where 
the evaulation and coevaulation maps 
are identified with $F(\on{ev}_X)$ and 
$F(\on{coev}_X)$ \cite[Exercise 2.10.6]{EGNO}.  
Let $F(X)$ be invertible in $\mc{C}'$, 
so that $F(\on{ev}_X)$ and 
$F(\on{coev}_X)$ are isomorphisms in 
$\mc{C}'$.  Since $F$ is exact and 
faithful, a morphism $f$ in $\mc{C}$ is 
an monomorphism (resp. epimorphism) if 
$F(f)$ is.  Since $F(\on{ev}_X)$ and 
$F(\on{coev}_X)$ are isomorphisms, 
$\on{ev}_X$ and $\on{coev}_X$ are 
isomorphisms as well and hence $X$ is 
invertible in $\mc{C}$.
\end{proof}

In nice cases, a one-sided inverse 
suffices for an object to be invertible.

\begin{lemma}
\label{yylem2.2}
Let $(\mc{C}, \otimes, \mathbbm{1})$ 
be a semisimple abelian monoidal category 
with finitely many simple objects up 
to isomophism.  If 
$V \otimes W \cong \mathbbm{1}$, then 
$W \otimes V \cong \mathbbm{1}$ and so 
$V$ and $W$ are invertible in $\mc{C}$.
\end{lemma}

\begin{proof}
The hypothesis implies that 
$F = V \otimes -$ and 
$G = W \otimes -$ satisfy 
$G \circ F \cong \on{id}$ as exact 
endofunctors of $\mc{C}$.  Choose 
representatives $S_1, \dots, S_n$ 
of the isomorphism classes of simple 
objects.  We have $S_i \cong G(F(S_i))$, 
and writing $F(S_i) = \bigoplus M_j$ 
with each $M_j$ simple, then $G(F(S_i)) = 
\bigoplus G(M_j)$, so there is some $j$ 
such that $G(M_j) = S_i$.  Thus we can 
choose some function 
$\rho: \{1,\dots, n\} \to \{1,\dots, n\}$ 
such that $S_i = G(S_{\rho(i)})$ for 
all $i$. Clearly $\rho$ must be injective 
and hence bijective. So $G$ acts as a 
permutation on the isomorphism classes 
of simple modules and $F$ acts by the 
inverse permutation.  Let $D_i = 
\End_{\mc{C}}(S_i)$ be the division ring 
of endomorphisms of each $S_i$.  For 
each $i$, applying $F$ to morphisms 
gives an endomorphism $f_i : D_i \to 
D_{\rho(i)}$ and similarly applying 
$G$ gives $g_i: D_{\rho_i} \to D_i$. 
By assumption $g_i \circ f_i = 1_{D_i}$.  
So $g_i$ is surjective, and it is also 
injective as its domain is a division 
ring.  Thus each $g_i$ is an isomorphism.  
Since $G$ is exact and all objects are 
direct sums of simples, it follows that 
$G$ is full and faithful and hence an 
autoequivalence.  Then $F$ must be a 
quasi-inverse to $G$ and we must have 
$F \circ G \cong \on{id}$ as well.  Thus 
$(W \otimes V) \otimes -$ is isomorphic 
to the identity functor, and necessarily 
$W \otimes V \cong \mathbbm{1}$ as well.  
\end{proof}

Next, we record some facts about invertible 
objects in the monoidal category $H \lMod$ 
for a weak Hopf algebra $H$, which follow 
quickly from the lemmas above. Recall that 
for each $M \in H \catMod$, $M$ has a 
natural underlying $H_t$-bimodule 
structure, and this gives a monoidal functor 
from $\mc{C} = (H\catMod, \otensor^{\ell}, H_t)$ 
to $\underline{\mc{C}} := ((H_t, H_t) \Bimod, \otimes_{H_t}, H_t)$.  In the proof below 
we denote the image of $M$ under this functor 
by $\underline{M}$.

\begin{lemma}
\label{yylem2.3}
Let $H$ be a weak Hopf algebra and keep the 
notation above.
\begin{enumerate}
\item[(1)]
$V \in H \catMod$ is invertible in $\mc{C}$ 
if and only if $\underline{V}$ is 
invertible in $\underline{\mc{C}}$.
\item[(2)] 
If $X,Y \in (H_t, H_t)\Bimod$ and 
$X \otimes_{H_t} Y \cong H_t$ as bimodules, 
then $X$ and $Y$ are invertible in 
$\underline{\mc{C}}$. 
\item[(3)]
If $V, W \in H\catMod$ and 
$V \otensor^{\ell} W \cong \mathbbm{1}$, 
then $V$ and $W$ are invertible in $\mc{C}$.
\end{enumerate}
\end{lemma}

\begin{proof}
(1) If $V$ is invertible in $\mc{C}$, 
$\underline{V}$ is invertible in 
$\underline{\mc{C}}$ by 
Lemma~\ref{yylem2.1}(1).  Conversely if 
$\underline{V}$ is invertible in 
$\underline{\mc{C}}$, then in particular 
it must be a $(H_t, H_t)$-bimodule which 
is finitely generated on both sides, so 
it is finite-dimensional over $\kk$. So 
$\dim_{\kk} V < \infty$ and hence $V$ has 
a left dual in $\mc{C}$. The functor 
$\underline{(-)}: \mc{C} \to 
\underline{\mc{C}}$ is clearly exact and 
faithful, so $V$ is invertible in $\mc{C}$ 
by Lemma~\ref{yylem2.1}(2).

(2) $(H_t, H_t) \Bimod \simeq H_t 
\otimes_{\kk} H_t^{op} \catMod$.  The 
algebra $H_t$ is separable over the base 
field $\kk$ in a weak Hopf algebra, so 
$H_t \otimes_{\kk} H_t^{op}$ is again 
semisimple. Thus the category 
$H_t \otimes_{\kk} H_t^{op} \catMod$
satisfies the hypotheses of 
Lemma~\ref{yylem2.2}, and the result 
follows.

(3) Since $V \otensor^{\ell} W \cong 
\mathbbm{1}$, applying the monoidal 
functor $\underline{(-)}$ we have 
$\underline{V} \otimes_{H_t} 
\underline{W} \cong H_t$.  Then 
$\underline{V}$ is invertible in 
$(H_t, H_t) \Bimod$ by (2), and so 
$V$ is invertible in $H \lMod$ by (1). 
\end{proof}

We also have the following useful property 
of invertible objects in $\rcatMod H$.

\begin{lemma}
\label{yylem2.4}
Let $H$ be a weak Hopf algebra and let 
$V$ be an invertible object in $(\rcatMod H, 
\otensor^{r}, \mathbbm{1})$, the monoidal 
category of right $H$-modules. Then 
$V \cong V^{S^2}$.
\end{lemma}

\begin{proof}
If $V \in \rcatMod H$ is invertible, it must 
have a left dual $V^* = \Hom_\kk(V, \kk)$ 
with the right action $(\phi \cdot h)(v) 
= \phi(vS(h))$. But since $V$ is invertible, 
it also has a right dual $^*V$ and 
$^* V \cong V^*$, both being the inverse 
of $V$. In particular, $V^*$ is also 
invertible with inverse $V$. But also 
since $V^*$ is invertible, its inverse 
must be its left dual $V^{**}$.  Thus 
$V^{**} \cong V$.
 
On the other hand, from the formula for 
$V^*$ we see that $V^{**} \cong V^{S^2}$.  
\end{proof}

\begin{lemma}
\label{yylem2.5}
Let $H$ be a weak Hopf algebra and let 
$U$ be an invertible $H$-bimodule.
\begin{enumerate} 
\item[(1)]
Let $W$ be an invertible object in 
$(H\lMod, \otensor^{\ell}, \mathbbm{1})$. 
Then both ${\mathbb F}^L(W)$ and 
${\mathbb F}^R(W)$ are invertible 
$H$-bimodules. As a consequence,
$W\otensor^{\ell} U$ and 
$U\otensor^{\ell} W$ are invertible 
$H$-bimodules.
\item[(2)]
Let $V$ be an invertible object in 
$(\rcatMod H, \otensor^{r}, \mathbbm{1})$. 
Then both ${\mathbb G}^L(V)$ and 
${\mathbb G}^R(V)$ are invertible 
$H$-bimodules. As a consequence,
$V\otensor^{r} U$ and $U\otensor^{r} V$
are invertible $H$-bimodules.
\end{enumerate}
\end{lemma}

\begin{proof} 
(1) By Lemma \ref{yylem1.3}(4), 
${\mathbb F}^L(W)$ is invertible in 
$(H, H) \Bimod$. Recall that 
${\mathbb F}^L(W)=W\otensor^{\ell} H$.  
Since $U$ is an $H$-bimodule, the 
isomorphism $W \otensor^{\ell} U 
\cong \mathbb{F}^L(W) \otimes_H U$ holds 
as $H$-bimodules, by 
Remark~\ref{yyrem1.5}.  The latter 
bimodule is a tensor product of two 
invertible bimodules and so is invertible.  
Similarly, ${\mathbb F}^R(W)$ is 
invertible and thus so is 
$U\otensor^{\ell} W$.

The proof of (2) is similar.
\end{proof}

In the rest of this section, we prove 
some elementary results about how 
invertible bimodules interact with 
twists by automorphisms.  These results 
are not specific to (weak) Hopf algebras.

If $A$ is an algebra, then it is 
well-known that an object $U \in A \Bimod$ 
has a left dual in the monoidal category 
$(A \Bimod, \otimes_A, A)$ if and only 
if it is finitely generated and projective 
as a left module, and in this case 
$U^* = \Hom_A(U, A)$.  Similarly the right 
dual exists if $U$ is finitely generated 
an projective on the right, and then 
$^* U = \Hom_{A^{\op}}(U, A)$. 
In particular, if $U$ is invertible, its 
inverse must be isomorphic to both 
$\Hom_A(U, A)$ and $\Hom_{A^{\op}}(U, A)$.

If $\tau, \sigma: A \to A$ are endomorphisms 
of an algebra $A$, and $M$ is an 
$A$-bimodule, then we write 
$\leftidx{^\tau} M^{\sigma}$ for the 
$A$-bimodule $M$ with actions 
$a \cdot m \cdot b = \tau(a)m\sigma(b)$. If 
either $\tau$ or $\sigma$ is the identity 
map it is omitted from the notation.

\begin{lemma}
\label{yylem2.6}
Let $A$ be a $\kk$-algebra with algebra 
endomorphism $\sigma: A \to A$.  Suppose 
that $A^{\sigma}$ is an invertible 
$A$-bimodule. Then $\sigma$ is an 
automorphism.
\end{lemma}

\begin{proof}
Write $U = A^{\sigma}$.  Since $U$ is 
invertible in the monoidal category of 
$A$-bimodules, the inverse of $U$ must 
be the left dual $U^*  = \Hom_A(U, A)$. 
Note that $U^* = \Hom(U,A) = 
\Hom_A(A^{\sigma}, A) \cong 
\leftidx{^{\sigma}} \!A$ as 
$A$-bimodules, via the map 
$f \mapsto f(1)$.  Because $U$ is 
invertible, there is also an isomorphism 
$U^* \otimes_A U \to A$. But 
$U^* \otimes_A U \cong \leftidx{^{\sigma}} 
\! A \otimes_A A^{\sigma} \cong 
\leftidx{^{\sigma}} \! A^{\sigma}$.  Fix 
an isomorphism $\phi: A \to 
\leftidx{^{\sigma}} \! A^{\sigma}$ of 
$A$-bimodules, and let $x = \phi(1)$.  
Then $\phi(z) = \phi(z \cdot 1) = z* x 
= \sigma(z) x$ and $\phi(z) = 
\phi(1 \cdot z) = x * z = x \sigma(z)$.  
Now if $z \in \ker(\sigma)$ then 
$\phi(z) = 0$, and since $\phi$ is 
injective, $z = 0$.  So $\sigma$ is 
injective.  Since $\phi$ is surjective 
there must be $y \in A$ such that 
$1 = \phi(y) = x \sigma(y) = \sigma(y) x$.  
So $x$ is a unit in $A$.  We also have 
$A = \phi(A) = \sigma(A) x$ since $\phi$ 
is an isomorphism of left $A$-modules. 
Let $a \in A$.  Then $ax = \sigma(b) x$ 
for some $b \in A$.  Since $x$ is a unit, 
$a = \sigma(b)$.  Thus $a \in \sigma(A)$ 
and $\sigma$ is surjective.
\end{proof}

Note that if $\sigma: A \to A$ is an 
automorphism, then the map 
$A \to A^{\sigma}$ given on the 
underlying sets by $\sigma$ is 
a right $A$-module isomorphism.  This 
observation also has a converse, under 
a quite weak hypothesis on the ring. We 
say that $A$ is \emph{Dedekind-finite} if 
every one-sided invertible element in $A$ 
is a unit \cite[Definition 2.2]{DGK}. It is 
well-known that every noetherian ring is 
Dedekind-finite. 

\begin{lemma}
\label{yylem2.7}
Let $A$ be a Dedekind-finite $\kk$-algebra 
with algebra endomorphism $\sigma: A \to A$. 
If $A^{\sigma} \cong A$ as right $A$-modules, 
then $\sigma$ is an automorphism.  
\end{lemma}

Recall that an algebra is said to be 
\emph{orthogonally finite} if it does not 
contain an infinite set of nonzero orthogonal 
idempotents \cite[Definition 2.1]{DGK}. It 
is clear that a left or right noetherian 
ring is orthogonally finite. By a result 
of Jacobson \cite{Jac50}, orthogonally 
finite rings are Dedekind-finite. 

\begin{lemma}
\label{yylem2.8} 
Suppose that $A$ is an orthogonally finite 
algebra. Let $M$ be an invertible 
$A$-bimodule and $\sigma,\tau$ be algebra 
endomorphisms of $A$. If ${^\sigma A}
\cong M^{\tau}$ as $A$-bimodules, then 
both $\sigma$ and $\tau$ are isomorphisms.
\end{lemma}

\begin{proof} Let $\phi: {^\sigma A}\to 
M^{\tau}$ be an isomorphism of $A$-bimodules 
and let $g=\phi(1)$. Then for 
$a\in {^\sigma A}$, 
$\phi(a)=\phi(1 a)=\phi(1) a=g \tau(a)$.
Thus $g$ is a generator of the right 
$A$-module $M$. So $M\cong A/N$ where 
$N = \rann(g)$. 

We claim that $N=0$. To see this note 
that $M$ is an invertible $A$-bimodule, 
so projective on both side; in particular,
$A/N$ is a projective right $A$-module. 
If $N\neq 0$, then $A\cong A/N\oplus N$ 
as a projective decomposition. So there 
is a nontrivial idempotent $e\in A$ such 
that $N=(1-e)A$ and $A/N\cong eA$. This 
means that $M_A$ is isomorphic to 
$eA$. It is well-known that 
$\End_{A^{op}}(eA)=eAe$. Since $M$ is an 
invertible $A$-bimodule, $f: A\cong 
\End_{A^{op}}(M)\cong \End_{A^{op}}(eA)=eAe$
is an isomorphism of $\kk$-algebras. Let 
$e_0=1$ be the identity element  of $A$, 
and define inductively $e_i=f(e_{i-1})$ 
for all $i\geq 1$, then 
$\{e_{i-1}-e_{i}\}_{i=1}^{\infty}$ is an 
finite set of orthogonal idempotents, 
yielding a contradiction. Therefore $N=0$ 
and consequently, $M\cong A$ as right
$A$-module. By Lemma \ref{yylem2.7} (and 
the fact that $A$ is Dedekind-finite), 
$\tau$ is an isomorphism.

Since $M$ is invertible and $\tau$ is an 
isomorphism, $M^{\tau}(\cong M\otimes_{A} 
A^{\tau})$ is invertible. Let $N$ be the 
inverse of $M^{\tau}$ as an $A$-bimodule.
Then $A\cong M^{\tau}\otimes_A N \cong
{^\sigma A}\otimes_A N \cong  {^\sigma N}$ as 
$A$-bimodules. By symmetry, $\sigma$ is an 
automorphism.
\end{proof}

\section{Homological Integrals in 
weak Hopf algebras}
\label{yysec3}

In this section, we will define the 
homological integral for an AS Gorenstein 
weak Hopf algebra and study some of its 
most basic properties; in particular we 
will show the integral is an invertible 
object in the appropriate monoidal 
category. We start by recalling some classical 
definitions that were mentioned in the introduction.

\begin{definition}
\label{yydef3.1}
Let $A$ be an algebra over $\kk$.  
\begin{enumerate}
\item[(1)] 
We say that $A$ satisfies the 
\emph{Van den Bergh condition} if 
\begin{enumerate}
\item[(1i)] 
$A$ has finite injective dimension $d$ 
as a left and right $A$-module; and 
\item[(1ii)] 
$\Ext^i_{A^e}(A, A^e) \cong 
\begin{cases} 0 & i \neq d \\ U & i 
= d \end{cases}$ for some invertible 
$A$-bimodule $U$. In this case, 
$U$ is called the \emph{Nakayama bimodule} 
of $A$. 
\end{enumerate}
\item[(2)] 
We say that $A$ is \emph{Artin--Schelter (AS)} Gorenstein if 
\begin{enumerate}
\item[(2i)] 
$A$ has finite injective dimension $d$ as 
a left and right $A$-module;  
\item[(2ii)] 
for every finite-dimensional left module 
$M$, $\Ext^i(M, A) = 0$ for $i \neq d$ 
and $\Ext^d(M, A)$ is a finite-dimensional 
right module; 
\item[(2iii)] 
the right sided analog of (2ii) also holds.
\end{enumerate}
\end{enumerate}  
\end{definition}

Suppose that $H$ is noetherian AS Gorenstein 
of injective dimension $d$. As a consequence 
of the definition, $\Ext^d_H(-, H)$ gives a 
duality between finite-dimensional left 
$H$-modules and finite-dimensional right 
$H$-modules, with inverse 
$\Ext^d_{H^{\op}}(-, H)$ \cite[Lemma 1.4]{RWZ1}.  
We will use this frequently below.

\begin{definition}
\label{yydef3.2}
Let $H$ be a weak Hopf algebra.  If $H$ is 
AS Gorenstein of injective dimension $d$, 
we define the \emph{left homological 
integral} of $H$ to be the right $H$-module 
$\int^{\ell}_H = \Ext^d_H(H_t, H)$. Similarly, 
the \emph{right homological integral} 
of $H$ is the left $H$-module 
$\int^r_H = \Ext^d_{H^{\op}}(H_s, H)$.  
More generally, for any such $H$ (not 
necessarily AS Gorenstein)  the 
\emph{left total integral} of $H$ 
is defined to be the right $H$-module 
$\widetilde{\int}^{\ell}_H =
\bigoplus_{s\geq 0} \Ext^s_H(H_t, H)$.  
Similarly, the \emph{right total integral} 
of $H$ is the left $H$-module 
$\widetilde{\int}^r_H = \bigoplus_{t\geq 0}
\Ext^t_{H^{\op}}(H_s, H)$.  
\end{definition}

It is clear that if $H$ is AS Gorenstein, then 
$\widetilde{\int}^{\ell}_H=\int^{\ell}_H$ 
and $\widetilde{\int}^r_H=\int^r_H$. If 
$H$ is not AS Gorenstein (for example, if 
$H$ is the Hopf algebra 
$\Bbbk\langle x,y\rangle$ with $x$ and $y$ 
primitive), both $\widetilde{\int}^{\ell}_H$ 
and $\widetilde{\int}^r_H$ can be infinite 
dimensional. 

If $H$ is a finite-dimensional weak Hopf 
algebra, then it is quasi-Frobenius and 
AS Gorenstein of dimension $0$, and 
$\Hom_H(H_t, H)$ can be identified with 
the space of left integrals, that is, 
$h \in H$ such that $gh = \epsilon_t(g)h$ 
for all $g \in H$. Similarly, 
$\Hom_{H^{\op}}(H_s, H)$ can be 
identified with the space of right 
integrals.  So we see that the definition 
of integrals for AS Gorenstein weak Hopf 
algebras generalizes the finite-dimensional 
case.

Our goal in the rest of this section is 
to prove that, when $H$ is an AS 
Gorenstein weak Hopf algebra, then the 
right homological integral is an 
invertible object in $H \lMod$. 
We need a few easy homological lemmas.

\begin{lemma}
\label{yylem3.3}
Let $R$ and $S$ be rings, $M \in 
R\lMod$, $N \in (R, S) \Bimod$, and 
$P \in (S, T) \Bimod$. Suppose that 
$P$ is flat as a left $S$-module.
Assume either that 
\begin{enumerate}
\item[(i)] \cite[Lemma 3.7(1)]{YZ3}
$M$ has a projective resolution by 
finitely generated projective 
$R$-modules, or 
\item[(ii)] 
$P$ is a finitely generated projective 
left $S$-module. 
\end{enumerate}
Then, for all $i \geq 0$, there is an 
isomorphism of right $T$-modules 
\[
\Ext^i_R(M, N \otimes_S P) \cong 
\Ext^i_R(M, N) \otimes_S P.
\]
\end{lemma}

\begin{proof}
(ii) There is a natural map $\phi: 
\Hom_R(M, N) \otimes_S P \to 
\Hom_R(M, N \otimes_S P)$ given by 
$\phi(f \otimes p)(m) = f(m) \otimes p$. 
This map is clearly an isomorphism when 
$P = S$, and therefore, when $P$ is a 
summand of a finitely generated free 
$S$-module. The result follows for 
$i \geq 0$ by calculating Ext using a 
projective resolution of $M$.
\end{proof}

\begin{lemma}
\label{yylem3.4}
Let $H$ be a weak Hopf algebra and let 
$W, V \in H \lMod$. Assume that $V$ is 
finite-dimensional and let $V^* \in H 
\lMod$ be its left dual. Then for every 
$i$, there is an isomorphism
\[
\Ext^i_H(W \otensor^{\ell} V, H) 
\xrightarrow{\cong} 
\Ext^i_H(W, H \otensor^{\ell} V^*)
\]
of right $H$-modules, where 
$H \otensor^{\ell} V^* = \mathbb{F}^R(V^*)$ 
as an $H$-bimodule.
\end{lemma}

\begin{proof}
We have the adjoint isomorphism that 
holds for any left dual in a monoidal 
category, 
\[
\phi: \Hom_H(W \otensor^{\ell} V, H) \to 
\Hom_H(W, H \otensor^{\ell} V^*).
\]
Because the right multiplication by 
$h \in H$ on $H \otensor^{\ell} V^*$ 
arises from applying the functor 
$- \otensor^{\ell} V^*$ to right 
multiplication by $h$ on $H$, it 
follows formally from the naturality 
of the adjoint isomorphism that $\phi$ 
is an isomorphism of right $H$-modules.
The result for Ext follows by taking a 
projective resolution of $W$, and using 
that $- \otensor^{\ell} V$ is an exact 
functor which preserves projective 
modules (see~\cite[Lemma 6.4(3)]{RWZ1}).
\end{proof}

\begin{lemma}
\label{yylem3.5}
Let $H$ be a weak Hopf algebra.  Suppose 
that $M, W \in H \lMod$, and that either 
$W$ is finite-dimensional or that $M$ has 
a projective resolution by finitely 
generated projective modules.
\begin{enumerate}
\item[(1)] 
Let $H \otensor^{\ell} W = \mathbb{F}^R(W)$ 
as $(H, H)$-bimodules. Then, for any $i \geq 0$,
\[
\Ext^i_H(M, H \otensor^{\ell} W) 
\cong \Ext^i_H(M, H) \otensor^r W^S
\]
as right $H$-modules.  
\item[(2)] 
For any $i$, we have $\Ext^i_H(-, H) \cong 
\Ext^i_H(H_t, H) \otensor^r ((-)^*)^S$, 
as functors from the category $H \lmod$ 
of finite-dimensional left $H$-modules to 
the category of right $H$-modules.
\end{enumerate}
\end{lemma}

\begin{proof}
(1) By Lemma~\ref{yylem1.8}(3), 
$H \otensor^{\ell} W = \mathbb{F}^R(W) 
\cong \mathbb{G}^R(W^S) = H \otensor^r W^S$ 
as $(H , H)$-bimodules. Since by 
definition by have $- \otensor^r W^S 
\cong - \otimes_H \mathbb{G}^R(W^S)$ as 
functors, we claim that 
\begin{gather*}
\Ext^i_H(M, H \otensor^{\ell} W) \cong 
\Ext^i_H(M, H \otensor^r W^S) \cong 
\Ext^i_H(M, H \otimes_H \mathbb{G}^R(W^S)) 
\cong \Ext^i_H(M, H) 
\otimes_H \mathbb{G}^R(W^S) \\
\cong \Ext^i_H(M, H) \otensor^r W^S,
\end{gather*}
using Lemma~\ref{yylem3.3} for the 
third isomorphism.  This is immediate 
from that lemma if $M$ has a 
projective resolution by finitely 
generated projective modules. In the 
other case, where $W$ is 
finite-dimensional, we need to show 
that the $H$-bimodule 
$\mathbb{G}^R(W^S)$ is finitely 
generated projective as a left 
$H$-module.  But as a left $H$-module, 
$H \otensor^r W^S = H \otimes_{H_s} W^S$ 
is a quotient of $H \otimes_k W^S$, 
which is finitely generated free on 
the left.

(2) Let $V$ be a finite-dimensional 
left $H$-module, and let $V^*$ be its 
left dual in the monoidal category 
$(H \lMod, \otensor^{\ell}, H_t)$. We 
have the right $H$-module isomorphism
\[
\Ext^i_H(V, H) \cong 
\Ext^i_H(H_t \otensor^{\ell} V, H) 
\cong \Ext^i_H(H_t, H \otensor^{\ell} V^*)
\]
from Lemma~\ref{yylem3.4}. Now 
$\Ext^i_H(H_t, H \otensor^{\ell} V^*) 
\cong \Ext^i_H(H_t, H) \otensor^r (V^*)^S$ 
as right $H$-modules, by part (1).  It is 
easy to check that the resulting isomorphism 
$\Ext^i_H(V, H) \cong \Ext^i_H(H_t, H) 
\otensor^r (V^*)^S$ is functorial in $V$.
\end{proof}

In some cases we may only know that a 
weak Hopf algebra $H$ is a direct sum 
of AS Gorenstein algebras, without 
understanding whether the coalgebra 
structure respects these summands.  In 
the next result, we see that this is 
still sufficient to prove that the total 
integral is well-behaved.

\begin{proposition}
\label{yypro3.6}
Let $H$ be a weak Hopf algebra which 
is a direct sum of AS Gorenstein 
algebras {\rm{(}}where we do not 
assume that this decomposition is a 
direct sum of AS Gorenstein weak Hopf 
algebras{\rm{)}}.
\begin{enumerate}
\item[(1)]
Let $V \in H \catMod$ be invertible in 
the monoidal category 
$(H \lMod, \otensor^{\ell}, H_t)$. Then  $\bigoplus_{s\geq 0} \Ext^s_H(V, H)$ is 
invertible in the monoidal category
$(\rMod H, \otensor^r, H_s)$.  In 
particular, the left total integral $\widetilde{\int}_H^{\ell}$ is 
invertible in $(\rMod H, \otensor^r, H_s)$.
\item[(2)]
Dually, $\widetilde{\int}_H^r$ is invertible 
in $(H \lMod, \otensor^{\ell}, H_t)$. 
\end{enumerate}
\end{proposition}

\begin{proof}
(1) Suppose that $H$ is a noetherian weak 
Hopf algebra and there is a direct sum 
decomposition $H=\bigoplus_{i=1}^{h} A_i$ 
as a direct sum of AS Gorenstein algebras. 
Let $e_i$ be the central idempotent 
corresponding to $A_i$ so $A_i=He_i$. Since 
$H$ is noetherian, so is each $A_i$.  Let 
$d_i$ denote the injective dimension of 
$A_i$ and define the functor 
${\mathbb E}_i = \bigoplus_{s\geq 0} 
\Ext^s_{A_i}(-, A_i)$. Then when viewed as 
a functor on the category of 
finite-dimensional left $A_i$-modules we 
have ${\mathbb E}_i 
\cong \Ext^{d_i}_{A_i}(-,A_i)$.  As 
mentioned earlier, $\Ext^{d_i}_{A_i}(-,A_i)$ 
(hence ${\mathbb E}_i$) is a contravariant 
equivalence from the category of 
finite-dimensional left $A_i$-modules to 
the category of finite-dimensional right 
$A_i$-modules. 

Let ${\mathbb E}$ be the functor 
$\bigoplus_{s\geq 0} \Ext^s_H(-,H)$.
Since $H$ is a direct sum of the $A_i$, 
therefore ${\mathbb E}\cong 
\bigoplus_{i=1}^{h}{\mathbb E}_i$ and it 
is a contravariant equivalence from the 
category of finite-dimensional left 
$H$-modules to the category of 
finite-dimensional right $H$-modules. In 
particular, the functor is essentially 
surjective and so there is a 
finite-dimensional left $H$-module $X$ 
such that ${\mathbb E}(X) \cong H_s$ as 
right $H$-modules. Now suppose that 
$V \in H\lMod$ is invertible with inverse 
$W$.  Then 
\begin{align*}
H_s &\cong {\mathbb E}(X) \cong 
\bigoplus_{s\geq 0} \Ext^s_H(X,H)
\cong \bigoplus_{s\geq 0} 
\Ext^s_H(V \otensor^{\ell} W 
\otensor^{\ell} X, H) \cong 
\bigoplus_{s\geq 0}
\Ext^s_H(V, H \otensor^{\ell} 
(W \otensor^{\ell} X)^*) \\
&\cong \bigoplus_{s\geq 0}
\Ext^s_H(V, H) \otensor^r 
((W \otensor^{\ell} X)^*)^S
\cong {\mathbb E}(V) 
\otensor^r (W \otensor^{\ell} X)^*)^S
\end{align*}
as right $H$-modules, where we have 
used Lemmas~\ref{yylem3.4} and 
\ref{yylem3.5}(1). This shows that 
${\mathbb E}(V) \otensor^r Y \cong H_s$ 
for some finite-dimensional right 
$H$-module $Y$. By a right-sided 
version of Lemma~\ref{yylem2.3}, 
${\mathbb E}(V)$ must be invertible in 
the category of right $H$-modules. The 
final statement follows by taking 
$V = H_t$, as the unit object is 
invertible in any monoidal category. 
The proof of (2) is similar.  
\end{proof}

Specializing to the AS Gorenstein case, 
we have the following.

\begin{proposition}
\label{yypro3.7}
Let $H$ be a noetherian AS Gorenstein 
weak Hopf algebra of injective dimension 
$d$.  Then the left homological integral 
$\int_H^{\ell} = \Ext^d_H(H_t, H)$ is 
invertible in $(\rMod H, \otensor^r, H_s)$. 
Similarly, the right homological integral 
$\int_H^r = \Ext^d_{H^{\op}}(H_s, H)$ is 
invertible in $(H \lMod, \otensor^{\ell}, 
H_t)$.  Moreover, $^S(\int_H^{\ell}) \cong 
\int_H^r$ as left $H$-modules and 
$\int_H^{\ell} \cong (\int_H^r)^S$ as 
right $H$-modules.
\end{proposition}

\begin{proof}
The first two statements are immediate 
from Proposition~\ref{yypro3.6}. Applying 
Lemma~\ref{yylem3.5}(2) to the left 
$H$-module $\int^r_H$ we get 
$\Ext^d_H(\int^r_H, H) \cong 
\int^{\ell}_H \otensor^r ((\int^{r}_H)^*)^S$.  
For any finite-dimensional left $H$-module 
$V$, it is easy to check that 
$(V^*)^S \cong (V^S)^*$ as right 
$H$-modules, where where $V^*$ is the 
left dual of $V$ in $(H \catMod, 
\otensor^{\ell})$ while $(V^S)^*$ is the 
left dual of the right module $V^S$ in 
$(\rcatMod H, \otensor^r)$. Now we have 
\[
\textstyle H_s \cong \Ext^d_H
(\Ext^d_{H^{op}}(H_s, H), H) 
= \Ext^d_H(\int^r_H, H) \cong 
\int^{\ell}_H \otensor^r 
((\int^{r}_H)^*)^S \cong 
\int^{\ell}_H \otensor^r ((\int^{r}_H)^S)^*
\]
as right $H$-modules.  Thus 
$((\int^{r}_H)^S)^*$ is the inverse of 
the invertible right module $\int^{\ell}_H$ 
in $(\rcatMod H, \otensor^r, H_s)$.  On 
the other hand every invertible module 
has its left dual as its inverse, so the 
inverse of $((\int^{r}_H)^S)^*$ is also 
$(\int^{r}_H)^S$.  Thus 
$\int^{\ell}_H \cong (\int^{r}_H)^S$ as 
right modules.  

For the other isomorphism, apply $^S(-)$ 
to both sides, giving 
$^S(\int^{\ell}_H) \cong 
{}^S((\int^{r}_H)^S) = {}^{S^2}(\int^{r}_H) 
\cong \int^{r}_H$, using Lemma~\ref{yylem2.4}.
\end{proof}

\begin{corollary}
\label{yylem3.8}
Let $H$ be a noetherian AS Gorenstein 
weak Hopf algebra. The following 
are equivalent:
\begin{enumerate}
\item[(1)] 
$\int^{\ell}_H \cong H_s$ as right 
$H$-modules.
\item[(2)]
$\int^r_H \cong H_t$ as left $H$-modules.
\end{enumerate}
\end{corollary}

\begin{proof}
This follows from the previous result, 
because $(H_t)^S \cong H_s$ and 
$^S(H_s) \cong H_t$ 
(see Lemma~\ref{yylem1.4}).
\end{proof}

\begin{definition}
\label{yydef3.9}
We say that $H$ is \emph{unimodular} if 
either of the equivalent conditions in 
the previous lemma holds.
\end{definition}

For an AS Gorenstein Hopf algebra $H$, 
unimodularity is equivalent to the statement 
that $\int^r_H \cong \kk \cong \int^{\ell}_H$ 
as $(H, H)$-bimodules. In our case, when 
$H$ is a weak Hopf algebra, $\int^r_H$ and 
$\int^{\ell}_H$ are not naturally 
$(H, H)$-bimodules, so this statement 
does not seem to have an analog.

\section{The integral and the 
Nakayama bimodule}
\label{yysec4}

In this section, we generalize the work 
of Brown and third-named author to 
calculate the Nakayama bimodule of an 
AS Gorenstein weak Hopf algebra. Along 
the way, we show that the antipode of 
an AS Gorenstein weak Hopf algebra is 
invertible under very general hypotheses, 
generalizing work of Skryabin, Le Meur, 
and Lu--Oh--Wang--Yu \cite{LeM, LOWY, Sk}.

\begin{definition}
\label{yydef4.1}
Let $H$ be a weak Hopf algebra. Let $\Delta': H \to H^e = H \otimes_{\kk} 
H^{\op}$ be defined by 
$\Delta'(h) = h_1 \otimes S(h_2)$. Then 
$\Delta'$ is multiplicative, so it is a 
(non-unital) ring homorphism. In particular 
$\Delta'(1) =  1_1 \otimes S(1_2)$ is 
idempotent in $H^e$.  Define a functor 
$L: H^e \lMod  \to H \lMod$ where 
$L(M) = \Delta'(1) M$ with left $H$ action 
given by pulling back the left $H^e$-action 
by $\Delta'$. In other words, thinking of 
$H^e \lMod$ as $(H,H) \Bimod$, the action 
of $H$ on $L(M)$ is given by 
$h \cdot m = h_1 m S(h_2)$.
\end{definition}

Note that since $\Delta'(H) \subseteq 
\Delta'(1) H^e \Delta'(1)$, $H^e \Delta'(1)$ 
is naturally an $(H^e, H)$-bimodule, where 
$H$ acts on the right via $\Delta'$.  
Similarly, $\Delta'(1) H^e$ is naturally 
an $(H, H^e)$-bimodule.  These bimodules 
can be better understood in terms of the 
functors in Section~\ref{yysec1}.

\begin{lemma}
\label{yylem4.2}
Let $H$ be a weak Hopf algebra. Then
\begin{enumerate}
\item[(1)] 
$H^e \Delta'(1) \cong \mathbb{F}^R(H) 
\cong \mathbb{G}^R( H^S)$ as modules in 
$(H^e, H) \Bimod = (H, H \otimes_{\kk} H) 
\Bimod$, where the additional right 
$H$-module structures on $\mathbb{F}^R(H)$ 
and $\mathbb{G}^R( H^S)$ come from 
Remark~{\rm{\ref{yyrem1.5}}}.  
\item[(2)] 
$\Delta'(1) H^e \cong \mathbb{F}^R
( {}^S H) \cong \mathbb{G}^R(H^{S^2})$ 
as modules in $(H, H^e) \Bimod = 
(H \otimes_{\kk} H, H) \Bimod$, where 
the additional left $H$-module structures 
on $\mathbb{F}^R( {}^S H)$ and 
$\mathbb{G}^R(H^{S^2})$ come from 
Remark~{\rm{\ref{yyrem1.5}}}. 
\end{enumerate} 
\end{lemma}

\begin{proof}
(1) We have 
\[
H^e \Delta'(1) = \{ g 1_1 \otimes 
S(1_2) h | g, h \in H \} 
= (H \otimes_{\kk}  H^S)\Delta(1) 
= H \otensor^r H^S = \mathbb{G}^R( H^S) 
\]
as $\kk$-spaces.  As $(H^e, H)$-bimodules 
the actions on $H^e \Delta'(1)$ are 
$(a \otimes c) \cdot (g \otimes h) \cdot b 
= agb_1 \otimes S(b_2)hc$.  Since $H^S$ 
is an $(H^{op}, H)$-bimodule with 
actions $x * h * y = S(y) h x$, 
$\mathbb{G}^R( H^S)$ has an additional 
left $H^{op}$-structure as in 
Remark~\ref{yyrem1.5}, and as such the 
actions on $H \otensor^r H^S$ are 
$(a \otimes c) \cdot (g \otimes h) \cdot b 
= agb_1 \otimes S(b_2) h c$.   This shows 
that $H^e \Delta'(1) \cong 
\mathbb{G}^R( H^S)$ as $H$-bimodules, 
where the isomorphism preserves the 
additional left $H^{op}$-action via $c$.

The isomorphism $\mathbb{G}^R( H^S) \cong 
\mathbb{F}^R(H)$ holds as 
$(H \otimes_{\kk} H^{op}, H)$-bimodules, 
equivalently as 
$(H, H \otimes_{\kk} H)$-bimodules, by 
Remark~\ref{yyrem1.9}.  For reference, 
the $(H, H \otimes_{\kk} H)$-bimodule 
structure on $\mathbb{F}^R(H) 
= H \otensor^{\ell} H$ is given explicitly 
by $a \cdot (g \otimes h) \cdot (b \otimes c) 
= a_1 g b \otimes a_2 h c$.

(2) This is similar to part (1).  We have 
\[
\Delta'(1) H^e = \{ 1_1 g \otimes 
h S(1_2) | g, h \in H \} 
= \Delta(1) (H \otimes_{\kk}  {}^S H)
= H \otensor^{\ell} ({}^S H) 
= \mathbb{F}^R( {}^S H). 
\]
The $(H, H \otimes_{\kk} H^{op})$-bimodule 
structure on $\Delta'(1) H^e$ is given 
by $a \cdot (g \otimes h) \cdot (b \otimes c) = 
a_1gb \otimes ch S(a_2)$.  On the other 
hand, ${}^S H$ is an $(H, H^{op})$-bimodule 
via $x \cdot h \cdot y = y h S(x)$ and 
so $\mathbb{F}^R({}^S H) 
= H \otensor^{\ell} ({}^S H)$ is an 
$(H, H \otimes H^{op})$-bimodule with actions 
$a \cdot (g \otimes h) \cdot (b \otimes c) 
= a_1 g b\otimes c h S(a_2)$.  This shows that 
$\Delta'(1) H^e \cong \mathbb{F}^R( {}^S H)$ 
as $(H, H^e)$-bimodules.  The further 
isomorphism $\mathbb{F}^R( {}^S H) \cong 
\mathbb{G}^R(H^{S^2})$ holds as 
$(H, H \otimes_{\kk} H^{op})$-bimodules, 
equivalently $(H \otimes_{\kk} H, H)$-bimodules, 
by Remark~\ref{yyrem1.9}. The 
$(H \otimes_{\kk} H, H)$-bimodule structure 
on $\mathbb{G}^R(H^{S^2}) 
= H \otensor^{r} H^{S^2}$ is given 
explicitly by $(a \otimes c) \cdot 
(g \otimes h) \cdot b 
=  agb_1  \otimes c hS^2(b_2)$.
\end{proof}

The functor $L$ turns out to be adjoint 
to a functor described earlier.  

\begin{lemma}
\label{yylem4.3}
Let $H$ be a weak Hopf algebra and retain the notation above.
The functor $L: H^e \lMod \to H \lMod$ is 
right adjoint to the functor 
$\mathbb{F}^L(-): H \lMod \to H^e \lMod$.
\end{lemma}

\begin{proof}
Recall that for any idempotent $e$ in a 
ring $R$ and $N \in R \lMod$ we have 
$\Hom_R(Re, N) \cong eN$ as left 
$eRe$-modules. In particular, for 
$M \in H^e \lMod$, $L(M) = \Delta'(1) M 
\cong \Hom_{H^e}(H^e\Delta'(1), M)$.
Since $L(-) = \Hom_{H^e}( H^e\Delta'(1), -)$ 
for the $(H^e, H)$-bimodule $H^e\Delta'(1)$, 
the left adjoint of $L$ is the functor 
$F = H^e \Delta'(1) \otimes_H (-)$.  To 
complete the proof we will show that 
$F \cong \mathbb{F}^L(-)$.

By Lemma~\ref{yylem4.2}(1), $H^e\Delta'(1) 
\cong \mathbb{F}^R(H)$ as $(H^e, H)$-bimodules.  
Thus for any $M \in H \lMod$, 
$F(M) = H^e \Delta'(1) \otimes_H M \cong 
\mathbb{F}^R(H) \otimes_H M 
\cong M \otensor^{\ell} H$, where the 
last isomorphism comes from the definition 
of $\mathbb{F}^R$.  Moreover, 
$M \otensor^{\ell} H \cong 
\mathbb{F}^L(M)$ by Lemma~\ref{yylem1.3}(2).  
This determines an isomorphism 
$\phi_M: F(M) \to \mathbb{F}^L(M)$ of 
$(H, H)$-bimodules which is clearly 
natural in $M$, as required.
\end{proof}

\begin{corollary}
\label{yycor4.4}
Let $H$ be a weak Hopf algebra and let $C$ be a $\kk$-algebra. Let 
$M \in (H^e, C) \Bimod$.  Let $H$ have 
its usual left $H^e$-structure. Then for all $i$,
there is a right $C$-module isomorphism 
$\Ext^i_{H^e}(H, M) \cong \Ext^i_H(H_t, L(M))$.
\end{corollary}

\begin{proof}
First of all, by Lemma \ref{yylem4.3}, 
we have the adjoint isomorphism 
$\Hom_{H^e}(\mathbb{F}^L(H_t), M) 
\cong \Hom_H(H_t, L(M))$. 
As we saw in the proof of that lemma, 
this can be interpreted as an instance 
of tensor-Hom adjointness, so it is 
clear that  $L(M) \in (H,C) \Bimod$ and 
that the adjoint isomorphism preserves 
this right $C$-module structure. By 
Lemma~\ref{yylem1.3}(4), 
$\mathbb{F}^L(H_t) \cong H$ as 
$H^e$-modules. The case $i = 0$ is proved.  

Now choose a projective resolution 
$P_{\bullet}$ of $H_t$ in $H \lMod$.  
The functor $L$ is clearly exact by 
its definition.   Because $\mathbb{F}^L$ 
is the left adjoint of an exact functor, 
$\mathbb{F}^L$ preserves projectives.  
Since $\mathbb{F}^L(-) \cong 
(-) \otensor^{\ell} H$, it is clear that 
$\mathbb{F}^L$ is also exact.  Thus 
$\mathbb{F}^L(P_{\bullet})$ is a 
projective resolution of $H$ in 
$H^e \lMod$. Using these projective 
resolutions to compute $\Ext$, we get 
an isomorphism $\Ext^i_{H^e}(H, M) 
\cong \Ext^i_H(H_t, L(M))$, and the 
right $C$-action is clearly preserved.
\end{proof}

The following is the main technical result 
that will allow us to relate the 
Van den Bergh condition and the AS 
Gorenstein condition.  

\begin{proposition}
\label{yypro4.5}
Let $H$ be a weak Hopf algebra.  Suppose 
that $H_t$ has a projective resolution 
$P_{\bullet}$ in $H \catMod$ such that 
every projective $P_i$ is finitely 
generated as a left $H$-module.
\begin{enumerate}
\item[(1)] 
For all $i \geq 0$ there are isomorphisms 
of right $H^e$-modules 
\[
\Ext^i_{H^e}(H, H^e) \cong 
\Ext^i_H(H_t, L(H^e)) \cong 
\Ext^i_H(H_t, H \otensor^r H^{S^2}) 
\cong \Ext^i_H(H_t, H) \otensor^r H^{S^2};
\]
here, identifying $\rMod H^e$ with 
$(H,H) \Bimod$ the right $H$-action is 
the normal one in the monoidal category 
$(\rMod H,\otensor^r, \mathbbm{1})$ while 
the left $H$-action is via the left side 
of $H^{S^2}$.
\item[(2)] 
Considering the right $H^e$-module 
$\Ext^i_{H^e}(H, H^e)$ as an 
$(H, H)$-bimodule, the functor 
$$- \otimes_H \Ext^i_{H^e}(H, H^e): 
\rMod H \to \rMod H$$ 
is naturally isomorphic to the 
functor $\Ext^i_H(H_t, H) 
\otensor^r (-)^{S^2}$.
\end{enumerate}
\end{proposition}

\begin{proof}
(1) The first isomorphism in the 
display is the adjoint isomorphism 
of Corollary~\ref{yycor4.4}, which 
as stated there maintains the 
additional right $H^e$-action because 
$H^e$ is a $(H^e, H^e)$-bimodule. The 
second isomorphism comes from the 
description of the $(H, H^e)$-bimodule 
structure of $L(H^e) = \Delta'(1) H^e$ 
as $\mathbb{G}^R(H^{S^2}) 
= H \otensor^r H^{S^2}$ using 
Lemma~\ref{yylem4.2}(2).  
Lemma~\ref{yylem3.5}(1) gives the 
third isomorphism as right $H$-modules, 
and it is straightforward to see that 
the right $H^{op}$-action, or equivalently 
the left $H$-action, via the left side of 
$H^{S^2}$ is preserved.

(2) Given $M \in \rMod H$, since the left 
$H$-action on $\Ext^i_H(H_t, H) \otensor^r 
H^{S^2}$ is via the left side of $H^{S^2}$, 
it is easy to see that 
\[
M \otimes_H (\Ext^i_H(H_t, H) 
\otensor^r H^{S^2}) 
\cong \Ext^i_H(H_t, H) \otensor^r 
(M \otimes_H H^{S^2}) 
\cong \Ext^i_H(H_t, H) 
\otensor^r M^{S^2}
\]
as right $H$-modules, and clearly this 
isomorphism is functorial in $M$. Now 
apply $M \otimes_H -$ to the isomorphism 
of part (1).
\end{proof}

Now we are ready to prove our first 
result about the bijectivity of the 
antipode.  

\begin{theorem}
\label{yythm4.6} 
Let $H$ be a noetherian weak Hopf algebra 
which is a finite sum of AS Gorenstein 
algebras.  Then the antipode $S$ is a 
bijection.
\end{theorem}

\begin{proof}
Let $U:= \bigoplus_{s\geq 0}
\Ext^{s}_{H^e}(H, H^e)$. Here, we are 
thinking of $H$ and $H^e$ as left 
$H^e$-modules; but since canonically 
$(H^e)^{\op} = H^e$, there is an 
equivalence of categories $H^e \lMod 
\to \rMod H^e$.  Thinking of $H$ and 
$H^e$ as right $H^e$-modules instead 
and applying this equivalence it is 
easy to see that 
$\bigoplus_{s\geq 0}
\Ext^{s}_{(H^e)^{\op}}(H, H^e) \cong U$ 
as well.  

Recall from Definition \ref{yydef3.2} 
that the left total integral of $H$ is 
$\widetilde{\textstyle \int}^{\ell}_H
:=\bigoplus_{s\geq 0}\Ext^s_{H}(H_t,H)$. 
By Proposition~\ref{yypro4.5}(1), we 
have 
\[
U = 
\bigoplus_{s\geq 0}\Ext^{s}_{H^e}(H, H^e)
\cong \bigoplus_{s\geq 0} 
\Ext^s_{H}(H_t,H)\otensor^{r} H^{S^2}
\cong \widetilde{\textstyle \int}^{\ell}_H 
\otensor^r (H)^{S^2}
\]
as $(H, H)$-bimodules, where the left 
$H$-module structure on 
$\widetilde{\textstyle \int}^{\ell}_H 
\otensor^r (H)^{S^2}$ is given by 
$g \cdot (v \otimes h) = v \otimes gh$. 
We could just as well have developed all 
of the same results on the other side. 
This leads to the symmetric result that 
$U \cong \leftidx{^{S^2}}(H) 
\otensor^{\ell} 
\widetilde{\textstyle \int}^{r}_H$ as 
bimodules, where the left $H$-module 
structure is the one coming from the 
monoidal structure on left $H$-modules, 
and the right module action on 
$\leftidx{^{S^2}}(H) \otensor^{\ell} 
\widetilde{\textstyle \int}^r_H$ is given 
by $(h \otimes v) \cdot g = hg \otimes v$.

We focus first on the right $H$-module 
side of the isomorphism 
$\widetilde{\textstyle \int}^{\ell}_H 
\otensor^r (H)^{S^2} 
\cong \leftidx{^{S^2}} \! (H) \otensor^{\ell} 
\widetilde{\textstyle \int}^r_H$.  Since 
$H$ is a direct sum of AS Gorenstein 
algebras, the right module 
$\widetilde{\textstyle \int}^{\ell}_H$ 
is invertible in the monoidal category 
of right $H$-modules, by 
Proposition~\ref{yypro3.6}(1).  Applying 
$(\widetilde{\textstyle \int}_H^{\ell})^{-1} 
\otensor^r -$ to both sides of our 
isomorphism yields
\[
H^{S^2} \cong 
\left(\widetilde{\textstyle \int}_H^{\ell}\right)^{-1} 
\otensor^r \left( \leftidx{^{S^2}} \! H 
\otensor^{\ell} \widetilde{\textstyle \int}_H^r\right)
\]
as $H$-bimodules. Now since 
$\widetilde{\textstyle \int}_H^r$ is 
an invertible object in the category of 
left $H$-modules 
(by Proposition~\ref{yypro3.6}(2)), we have 
$\leftidx{^{S^2}}{} 
\left(\widetilde{\textstyle \int}_H^r\right) 
\cong \widetilde{\textstyle \int}_H^r$ as 
left modules by Lemma~\ref{yylem2.4} and 
hence
\[
\leftidx{^{S^2}\! H} \otensor^{\ell} 
\widetilde{\textstyle \int}_H^r 
\cong \left(\leftidx{^{S^2}} \! H\right) 
\otensor^{\ell} 
\left(\textstyle {\leftidx{^{S^2}} \! 
\widetilde{\textstyle \int}_H^r}\right) \cong 
\leftidx{^{S^2}} \! 
\left(H \otensor^{\ell} 
\widetilde{\textstyle \int}_H^r\right), 
\]
since twisting by $S^2$ is a monoidal 
functor (for example, by applying  
Lemma~\ref{yylem1.1} twice). These are 
actually isomorphisms of bimodules, where 
the right $H$-action remains via the right 
side of the first tensorand $H$. Because 
$\otensor^r$ is the tensor product in the 
right $H$-module category, we have 
\[
H^{S^2} \cong 
\left(\widetilde{\textstyle \int}_H^{\ell}\right)^{-1} 
\otensor^r \leftidx{^{S^2}} \! 
\left( H \otensor^{\ell} 
\widetilde{\textstyle \int}_H^r\right)
\cong \leftidx{^{S^2}}{} 
\left[\left(\widetilde{\textstyle \int}_H^{\ell}\right)^{-1} 
\otensor^r \left( H \otensor^{\ell} 
\widetilde{\textstyle \int}_H^r\right)\right] 
\]
as $H$-bimodules. Because 
$\widetilde{\textstyle \int}_H^{\ell}$ 
and $\widetilde{\textstyle \int}_H^r$ 
are invertible objects, applying Lemma 
\ref{yylem2.5} twice, 
$(\widetilde{\textstyle \int}_H^{\ell})^{-1} 
\otensor^r ( H \otensor^{\ell} 
\widetilde{\textstyle \int}_H^r)$ is an 
invertible $H$-bimodule. By 
Lemma \ref{yylem2.8}, $S^2$ is an isomorphism, 
whence $S$ is bijective.
\end{proof}

We can now show that the Van den Bergh 
condition and the AS Gorenstein 
properties for a weak Hopf algebra are 
usually simultaneously satisfied.

\begin{theorem}
\label{yythm4.7} 
Let $H$ be a weak Hopf algebra where $H$ 
has injective dimension $d$ as a left and 
right $H$-module. Assume that $H_t$ has a 
projective resolution by finitely 
generated projective left $H$-modules.
\begin{enumerate}
\item[(1)] 
Suppose that $H$ satisfies the Van den 
Bergh condition. Then $H$ is AS Gorenstein 
and the antipode $S$ is a bijection.
\item[(2)] 
Suppose that $H$ is noetherian AS Gorenstein. 
Then $H$ satisfies the Van den Bergh 
condition and the antipode $S$ is a bijection.
\end{enumerate}
\end{theorem}

\begin{proof}
(1) We have that $\Ext^i_{H^e}(H, H^e)$ 
is zero for $i \neq d$ and an invertible 
bimodule $U$ when $i = d$.  By 
Proposition~\ref{yypro4.5}(2), 
$- \otimes_H \Ext^i_{H^e}(H, H^e)$ (which 
is a functor from $\rcatMod H \to \rcatMod H$)  
is naturally isomorphic to the functor 
$\Ext^i_H(H_t, H) \otensor^r (-)^{S^2}$. 
For $i \neq d$ this shows that 
$\Ext^i_H(H_t, H) \otensor^r (-)^{S^2}$ is 
the $0$ functor. Applying this to $H_s$, 
since $H_s^{S^2} \cong H_s$ by 
Lemma~\ref{yylem2.4}, we see that 
$\Ext^i_H(H_t, H) = 0$. When $i = d$, since 
$U = \Ext^d_{H^e}(H, H^e)$ is an invertible 
bimodule we get that $- \otimes_H U$ is an 
autoequivalence, so 
$\int_H^{\ell} \otensor^r (-)^{S^2}$ is an 
autoequivalence, where $\int_H^{\ell} 
= \Ext^d_H(H_t, H)$. In particular, it is 
essentially surjective, so there is 
$V \in \rMod H$ such that 
$\int_H^{\ell} \otensor^r V^{S^2} 
\cong H_s$ as right $H$-modules.  By a 
right $H$-module version of 
Lemma~\ref{yylem2.3}, $\int_H^{\ell}$ is 
invertible.  The autoequivalence 
$\int_H^{\ell} \otensor^r (-)^{S^2}$ is 
a composition of $(-)^{S^2}$ and 
$\int_H^{\ell} \otensor^r (-)$, and 
since $\int_H^{\ell}$ is invertible the 
latter functor is also an autoequivalence.  
Thus $(-)^{S^2} : \rMod H \to  \rMod H$ 
is an autoequivalence, or in other words 
$- \otimes_H H^{S^2}$ is an autoequivalence 
and hence $H^{S^2}$ is an invertible 
$(H, H)$-bimodule.  By Lemma~\ref{yylem2.6}, 
$S^2$ must be an automorphism of $H$. 
Thus $S$ is a bijection.   

For any finite-dimensional left 
$H$-module $V$, we have 
$\Ext^d_H(V, H) \cong \int_H^{\ell} 
\otensor^r (V^*)^S$, by 
Lemma~\ref{yylem3.5}(2).  Now we know 
$\int_H^{\ell}$ is invertible and hence 
it must be a finite-dimensional right 
$H$-module. Since $(V^*)^S$ is clearly 
also finite-dimensional, so is 
$\int_H^{\ell} \otensor^r (V^*)^S$. Thus 
$\Ext^d_H(V, H)$ is finite-dimensional 
for all finite-dimensional $V$.

We have verified all of the left-sided 
conditions in the definition of AS 
Gorenstein. Since $H$ satisfies the 
Van den Bergh condition, so does 
$H^{\op, \cop}$, where $H^{\op, \cop}$ 
is also a weak Hopf algebra with 
antipode $S$.  So the right module 
conditions in the definition of AS 
Gorenstein also hold.  

(2) By Theorem \ref{yythm4.6}, $S$ is 
bijective. Assume that $H$ is AS 
Gorenstein of dimension $d$.  By 
Proposition~\ref{yypro3.7}, the left 
integral $\int_H^{\ell} = 
\Ext^d_H(H_t, H)$ is an invertible right 
$H$-module. Let $U = \Ext^d_{H^e}(H, H^e)$. 
By Proposition~\ref{yypro4.5}(2), 
there is an isomorphism 
$- \otimes_H U \cong \int_H^{\ell} 
\otensor^r (-)^{S^2}$ of 
functors $\rMod H \to \rMod H$.  
Similarly as in part (1), the latter 
functor is a composition of two functors 
$(-)^{S^2}$ and $\int^{\ell}_H 
\otensor^r (-)$, the first of which is 
an autoequivalence since $S$ is a 
bijection, and the second of which is 
an autoequivalence since $\int_H^{\ell}$ 
is invertible.  Thus $- \otimes_H U$ is 
an autoequivalence, which forces $U$ to 
be an invertible $(H, H)$-bimodule.  On 
the other hand, for $i \neq d$ we have 
$- \otimes_H \Ext^i_{H^e}(H, H^e): 
\rcatMod H \to \rcatMod H$ is naturally 
isomorphic to the functor 
$\Ext^i_H(H_t, H) \otensor^r (-)^{S^2}$, 
which is $0$ since $\Ext^i_H(H_t, H) = 0$. 
Then $\Ext^i_{H^e}(H, H^e) = 0$.  So $H$ 
satisfies the Van den Bergh condition.
\end{proof}

\begin{corollary}
\label{yycor4.8}
Let $H$ be a noetherian weak Hopf algebra. 
Then $H$ is AS Gorenstein if and only if 
$H$ satisfies the Van den Bergh condition, 
and in either case the antipode $S$ is 
bijective.
\end{corollary}

Suppose now that $H$ is a weak Hopf 
algebra of injective dimension $d$ which 
satisfies the Van den Bergh condition. 
Recall that the invertible bimodule 
$U = \Ext^d_{H^e}(H, H^e)$ is called 
the \emph{Nakayama bimodule} 
[Definition \ref{yydef3.1}].  When $H$ 
is a Hopf algebra, then $U \cong 
\prescript{1}{} H^{\mu}$ as an 
$(H, H)$-bimodule, with 
$\mu = \xi \circ S^2$, where $\xi$ is 
the left winding automorphism 
associated to the grouplike element 
given by the left homological integral 
\cite[Section 4.5]{BZ}.  In the weak 
case, because $\int_H^{\ell}$ may not 
be a free right $H_s$-module, $U$ may 
not be a free right $H$-module, so we 
cannot necessarily express $U$ in the 
form $U \cong \prescript{1}{} H^{\mu}$ 
as an $(H, H)$-bimodule.  So there may 
be no Nakayama automorphism $\mu$ in 
the traditional sense, but the 
isomorphism $U \cong \int^{\ell}_H 
\otensor^r (H)^{S^2}$ we found above is 
a clear generalization of the concept, 
which serves the same purpose. For example, 
we can find a formula for the powers of 
the Nakayama bimodule in terms of the 
integral.

\begin{theorem}  
\label{yythm4.9}
Let $H$ be a weak Hopf algebra of 
injective dimension $d$ which satisfies 
the Van den Bergh condition, and let 
$U:= \Ext^d_{H^e}(H, H^e)$ be the 
Nakayama bimodule. Then for all 
$n \geq 1$ we have an isomorphism of 
functors 
$- \otimes_H U^{\otimes n} \cong 
((\int^{\ell}_H)^{\otensor^r n}) \, 
\otensor^r (-)^{S^{2n}}$. 
In particular, 
$U^{\otimes n} \cong 
((\int^{\ell}_H)^{\otensor^r n}) \,
\otensor^r H^{S^{2n}}$.
\end{theorem}

\begin{proof}
By Theorem \ref{yythm4.7}(1), we know 
that $S$ is a bijection. We claim that 
the autoequivalences $(-)^{S^2}$ and 
$\int_H^{\ell} \otensor^r (-)$ commute 
with each other in the group of 
autoequivalences of $\rcatMod H$, up to 
isomorphism of functors.  First note that 
$(\int^{\ell}_H \otensor^r M)^{S^2} \cong 
(\int^{\ell}_H)^{S^2} \otensor^r M^{S^2}$ 
because $S^2$ is a coalgebra map.  But 
since $\int^{\ell}_H$ is invertible, we get 
$(\int^{\ell}_H)^{S^2} \cong \int^{\ell}_H$ 
by Lemma~\ref{yylem2.4}. Thus 
$(\int^{\ell}_H \otensor M)^{S^2} 
\cong \int^{\ell}_H \otensor M^{S^2}$, 
and this isomorphism is natural in $M$, 
proving the claim. Now since $- \otimes_H U$ 
is isomorphic to the composition 
$(\int^{\ell}_H \otensor (-)) 
\circ (-)^{S^2}$, taking $n$th powers the 
result follows.
\end{proof}

\section{Module-finite WHAs are direct 
sums of AS Gorenstein WHAs}
\label{yysec5}

As we saw in the preceding section, knowing 
that a weak Hopf algebra $H$ is a direct 
sum (as algebras) of finitely many AS 
Gorenstein algebras is already sufficient 
to prove some interesting results using 
total homological integrals.  Still, it 
is certainly more convenient if we know 
that a  weak Hopf algebra is a direct sum 
\emph{as weak Hopf algebras} of finitely 
many AS Gorenstein weak Hopf algebras.  
In this section we prove that if $H$ is a 
weak Hopf algebra which is finite over an 
affine center, then we can obtain this 
nicer conclusion.

Let $A$ be a $\kk$-algebra.  We say that 
a left $A$-module $M$ is 
\emph{residually finite} (equivalently, 
\emph{residually finite-dimensional} as 
in \cite[Definition 9.2.9]{Mo}) if
\[
\bigcap \{ M' \mid M'\ 
\text{is an $A$-submodule of}\ M\ 
\text{with}\ \dim_\kk M/M' 
< \infty \} = 0.
\]   
This is equivalent to the natural map 
$M \to \prod_{M'} M/M'$ being an 
injection, where the product is over 
the set of $M'$ above. In particular, 
the algebra $A$ is called (left) 
\emph{residually finite} if $A$ is 
residually finite as a left $A$-module. 
The following lemma is clear.

\begin{lemma}
\label{yylem5.1}
The following hold.
\begin{enumerate}
\item[(1)]
Every locally finite ${\mathbb N}$-graded 
algebra is residually finite. In particular, 
the commutative polynomial ring 
$\Bbbk[x_1,\cdots,x_n]$ is residually finite.
\item[(2)]
If $A_1, A_2, \dots, A_n$ are residually 
finite algebras, then so is the direct sum 
$A_1\oplus A_2\oplus \cdots \oplus A_n$.
\end{enumerate}
\end{lemma}

\begin{proposition}
\label{yypro5.2}
Let $H$ be a weak Hopf algebra that is a 
finite module over its affine center. Then 
$H$ is residually finite.
\end{proposition}

\begin{proof} By \cite[Theorem 0.3(1)]{RWZ1}, 
$H$ is a direct sum $\bigoplus_{i=1}^n H_i$, 
of indecomposable noetherian algebras which 
are AS Gorenstein, Auslander Gorenstein, 
Cohen--Macaulay, and homogeneous of finite 
Gelfand--Kirillov dimension equal to their 
injective dimension (see \cite{RWZ1} for 
definitions). By Lemma \ref{yylem5.1}(2), it 
remains to show that each $H_i$ is residually 
finite.

Fix some $1 \leq i \leq n$ and let $A = H_i$. 
Since $H$ is affine, so is $A$. Let $Z$ be 
the center of $A$, which is affine, since 
$Z(H)$ is affine. Now by the Noether 
Normalization Theorem, $Z$ contains a 
subalgebra $C$, such that $Z$ is a finite 
module over $C$ and $C$ is isomorphic to a 
polynomial ring $\Bbbk[x_1,\cdots,x_n]$. 
By \cite[Theorem 0.3(1)]{RWZ1}, $A$ satisfies 
\cite[Hypothesis 4.1]{RWZ1}, and clearly, 
$C$ satisfies \cite[Hypothesis 4.1]{RWZ1}. By 
\cite[Lemma 4.4]{RWZ1}, $H_i$ is a projective 
module over $C$, therefore free over $C$.

Since $C$ is a polynomial ring, there is a 
sequence of co-finite-dimensional ideals 
$\{I_{s}\}_s$ such that $\bigcap_s I_s=0$.
Consider the short exact sequence
$$0\to I_s\to C\to C/I_s\to 0,$$
which induces a short exact sequence
$$0\to I_s\otimes_{C} A\to A\to C/I_s
\otimes_C A\to 0.$$
Since $A$ is a free module over $C$, 
therefore $\bigcap_{s}I_s\otimes_{C} A=0$. 
Since each $C/I_s\otimes_C A$ is finite 
dimensional, $I_s\otimes_C A$ is a 
co-finite-dimensional ideal of $A$. Thus 
$A$ is residually finite as required.
\end{proof}

\begin{lemma}
\label{yylem5.3}
Let $H$ be a weak Hopf algebra over $\kk$. 
Suppose that $H = H_1 \oplus H_2 \oplus 
\dots \oplus H_m$ as algebras, where 
$H_i = e_i H$ for a set 
$\{ e_i \mid 1 \leq i \leq m \}$ of central 
pairwise orthogonal idempotents with 
$1 = e_1 + e_2 + \dots + e_m$.
\begin{enumerate}
\item[(1)] 
Suppose that $\Delta(H_i) \subseteq H_i 
\otimes_\kk H_i$ and $S(H_i) \subseteq H_i$ 
for all $i$. Then each $H_i$ is a weak 
Hopf algebra over $\kk$, and so 
$H = H_1 \oplus H_2 \oplus \dots \oplus H_m$ 
as weak Hopf algebras.
\item[(2)] 
Assume that $H$ is residually finite.  
Suppose that for all $i$ and for all 
finite-dimensional modules $V \in H_i \lMod$, 
$W \in H_j \lMod$ we have 
$V \otensor^{\ell} W = 0$ if $i \neq j$; 
$V \otensor^{\ell} W \in H_i \lMod$ if 
$i = j$; and $V^* \in H_i \lMod$.  Then 
the hypotheses of part {\rm{(1)}} hold 
and  $H = H_1 \oplus H_2 \oplus 
\dots \oplus H_m$ as weak Hopf algebras.
\end{enumerate}
\end{lemma}

\begin{proof}
(1)  The hypothesis implies that it 
makes sense to ask whether the algebra 
$H_i$ (with unit element $e_i$) is a 
weak Hopf algebra with coproduct 
$\Delta_i = \Delta \vert_{H_i}: H_i 
\to H_i \otimes_\kk H_i$, counit 
$\epsilon_i=\epsilon\vert_{H_i}: 
H_i \to \kk$ and antipode 
$S_i = S \vert_{H_i}: H_i \to H_i$. 
The axioms of a coalgebra for $H$ 
immediately restrict to give 
that $(H_i, \Delta_i, \epsilon_i)$ 
is a coalgebra.  Similarly, 
$\Delta_i(gh) = \Delta_i(g)\Delta_i(h)$ 
for $g, h \in H_i$ and 
$\epsilon_i(fgh) 
= \epsilon_i(fg_1)\epsilon_i(g_2h) 
= \epsilon_i(fg_2) \epsilon_i(g_1h)$ 
for $f, g, h \in H_i$ since these are 
properties of $H$ and we are restricting.  

Now by assumption, for each 
$1 \leq i \leq m$, we have 
$\Delta(e_i) \in e_i H \otimes_{\kk} e_i H 
= (e_i \otimes e_i)(H \otimes H)$, 
say $\Delta(e_i) = (e_i \otimes e_i)\Omega_i$ 
with $\Omega_i \in H \otimes H$. Since 
$\Delta(1) = \sum_{i=1}^m \Delta(e_i) 
= \sum_{i=1}^m (e_i \otimes e_i) \Omega_i$, 
we see that $(e_i \otimes e_i) \Delta(1) 
= (e_i \otimes e_i) \Omega_i = \Delta(e_i)$.  
Similarly, $\Delta^2(e_i) \in 
(e_i \otimes e_i \otimes e_i)(H^{\otimes 3})$ 
implies that $\Delta^2(e_i) 
= (e_i \otimes e_i \otimes e_i)(\Delta^2(1))$. 
Now multiplying both sides of the identity 
\[
\Delta^2(1) = (\Delta(1) \otimes 1)(1 \otimes \Delta(1))
= (1 \otimes \Delta(1))(\Delta(1) \otimes 1)
\]
by $e_i \otimes e_i \otimes e_i$ gives the 
required identity
\[
\Delta_i^2(e_i) = 
(\Delta_i(e_i) \otimes e_i)(e_i \otimes \Delta_i(e_i)) 
= (e_i \otimes \Delta_i(e_i))(\Delta_i(e_i) \otimes e_i).
\]

Thus $H_i$ is a weak bialgebra with the 
given operations. Note that the map 
$(\epsilon_t)_i$ for this algebra is given 
for $h \in H_i$ by 
\[
(\epsilon_t)_i(h) = 
\epsilon_i(\Delta_i(e_i)_1 h)) \Delta_i(e_i)_2 
= \epsilon(e_i 1_1 h) e_i 1_2 = e_i \epsilon(1_1 e_i h)1_2  
= e_i \epsilon(1_1 h) 1_2 = e_i \epsilon_t(h).
\]
Similarly, $(\epsilon_s)_i(h) = e_i \epsilon_s(h)$. 
Thus, for example, since we have 
$h_1 S(h_2) = \epsilon_t(h)$ in $H$, multiplying 
by $e_i$ on both sides and using that 
$\Delta(h) \in e_iH \otimes e_iH$ gives 
\[
\Delta_i(h)_1 S_i(\Delta_i(h)_2) 
= h_1 S(h_2) = e_i h_1 S(h_2) 
= e_i \epsilon_t(h) = (\epsilon_t)_i(h).
\]
The other two requires properties of the 
antipode follow similarly. So each $H_i$ is 
a weak Hopf algebra.

Finally, recall that the weak Hopf algebra 
structure on the direct sum 
$H_1 \oplus H_2 \oplus \dots \oplus H_m$ is 
defined by coordinatewise coproduct and 
antipode, and counit $\epsilon(h_1, \dots, h_m) 
= \sum_{i=1}^m \epsilon_i(h_i)$. It is clear 
that this is the same as the original weak 
Hopf algebra structure on $H$.

(2)  Recall that $V\in H_i\lMod$ means 
that $e_j V=0$ for all $j\neq i$. First we 
show that the finite-dimensional modules 
$V$ and $W$ can be replaced by any 
\emph{residually} finite modules in the 
first two hypotheses.  So suppose that 
$M \in H_i \lMod$ and $N \in H_j \lMod$ are 
residually finite. Take any left $H$-submodules 
$M' \subseteq M$ and $N' \subseteq N$ with 
$\dim_{\kk} M/M' < \infty$ and 
$\dim_{\kk} N/N' < \infty$. Then $(M/M') 
\otensor^{\ell} (N/N') = 0$ if $i \neq j$, and 
$(M/M') \otensor^{\ell} (N/N') \in H_i \lMod$ 
if $i = j$, by hypothesis.  Note that there 
is a short exact sequence
\begin{equation}
\label{E5.3.1}\tag{E5.3.1}
0 \to M' \otimes_\kk N + M \otimes_\kk N' 
\to M \otimes_\kk N \overset{\phi}{\to} M/M' 
\otimes_\kk N/N' \to 0.
\end{equation}
Given any $x \in M \otimes_{\kk} N$, there 
are finite-dimensional subspaces 
$M'' \subseteq M$ and $N'' \subseteq N$ 
such that $x \in M'' \otimes_{\kk} N''$. 
Since the intersection of all submodules 
$M'$ such that $M/M'$ is finite-dimensional 
is $0$, we can choose such an $M'$ with 
$M' \cap M'' = 0$. Similarly, choose a 
submodule $N'$ of $N$ with 
$\dim_{\kk} N/N' < \infty$ and 
$N' \cap N'' = 0$. By construction, 
$0 \neq \phi(x) \in M/M' \otimes_{\kk} N/N'$ 
above, and thus $x \not \in M' 
\otimes_{\kk} N + M \otimes_{\kk} N'$. 
It follows that $\bigcap_{M', N'} 
(M' \otimes_{\kk} N + M \otimes_{\kk} N') = 0$, 
where the intersection is over all $M', N'$ 
of the form above.  

Now multiplying \eqref{E5.3.1} above on the 
left by $\Delta(1)$ we have an exact sequence
\[
0 \to M' \otensor^{\ell} N + M \otensor^{\ell} N' 
\to M \otensor^{\ell} N 
\to M/M' \otensor^{\ell} N/N' \to 0.
\]
We also have $\bigcap_{M', N'} 
(M' \otensor^{\ell} N + M \otensor^{\ell} N') = 0$, 
since it is a subset of the intersection above.  
This shows that we have an injective map 
\[
M \otensor^{\ell} N 
\to \prod_{M', N'} M/M' \otensor^{\ell} N/N'
\]
and in particular that $M \otensor^{\ell} N$ 
is again residually finite.  Now when 
$i \neq j$ the right hand side is $0$, so 
$M \otensor^{\ell} N = 0$.  Similarly, when 
$i = j$ the right hand side is in $H_i \lMod$, 
so $M \otensor^{\ell} N \in H_i \lMod$.

By hypothesis, $H$ is residually finite, so 
each module $H_i$, being a factor module, is 
also residually finite. Applying the result 
of the previous paragraph gives 
$H_i \otensor^{\ell} H_j = 0$ for $i \neq j$ 
and $H_i \otensor^{\ell} H_i \in H_i \lMod$.

We now show how this implies the first 
hypothesis of part (1). Fix $i$ and write 
$\Delta(e_i) = \sum_{r,s = 1}^m \Omega_{r,s}
(e_r \otimes e_s)$ for some 
$\Omega_{r,s} \in e_r H \otimes e_s H$. Then 
\[
e_i \cdot (H_p \otensor^{\ell} H_q) 
= \Delta(e_i) 
\Delta(1)(H_p \otimes_{\kk} H_q) 
= \Delta(e_i)(H_p \otimes_{\kk} H_q) 
= \sum_{r,s} \Omega_{r,s} (e_r \otimes e_s)
(H_p \otimes_{\kk} H_q) 
= \Omega_{p,q} (H_p \otimes_{\kk} H_q).
\]
If $p \neq q$ or if $p = q \neq i$, then 
$e_i \cdot (H_p \otensor^{\ell} H_q) = 0$ 
and so $\Omega_{p,q} 
= \Omega_{p,q}(e_p \otimes e_q) \in \Omega_{p,q} 
(H_p \otimes_{\kk} H_q) = 0$, so 
$\Omega_{p,q} =0$.  Thus $\Delta(e_i) 
= \Omega_{i,i}(e_i \otimes e_i) \in e_i H 
\otimes_{\kk} e_i H = H_i \otimes_{\kk} H_i$.

Now consider the hypothesis that if 
$V \in H_i \lMod$ is finite-dimensional, 
then $V^* \in H_i \lMod$; we use this to 
show the second hypothesis of (1) that 
$S(H_i) \subseteq H_i$ for all $i$.  Given 
$v \in V$, $\phi \in V^*$, and some $e_j$ 
we have $[e_j \phi](v) = \phi(S(e_j)v)$ by 
definition.  If $j \neq i$ then $e_j V^* = 0$ 
so $e_j \phi = 0$, and thus $\phi(S(e_j)v) = 0$.  
Since this holds for all $\phi \in V^*$, we 
conclude that $S(e_j)v = 0$.  Since this holds 
for all $v \in V$, we have $S(e_j) V = 0$.  
If $M$ is a residually finite $H_i$-module, 
then $M$ embeds in a product of 
finite-dimensional $H_i$-modules, so we 
conclude that $S(e_j) M = 0$ as well.  Since 
$H_i$ itself is residually finite, as noted 
above, we have $S(e_j) H_i = 0$. In particular, 
$S(e_j) e_i = 0$ if $i \neq j$.  This shows 
that $S(e_i) \in e_i H$ for any $i$, so $S(e_i H) 
\subseteq e_i H$ and we are done.
\end{proof}

Now we are ready to prove the main result of 
this section.

\begin{theorem}
\label{yythm5.4}
Let $H$ be a noetherian weak Hopf algebra
such that $H$ is residually finite as a 
left $H$-module. Assume in addition that 
$H \cong K_1 \oplus \dots \oplus K_n$ as 
algebras, where each $K_i$ is an AS 
Gorenstein algebra. Then $H \cong H_1 
\oplus \dots \oplus H_m$ as weak Hopf 
algebras, where each $H_i$ is an AS 
Gorenstein weak Hopf algebra. 
\end{theorem}

\begin{proof}
By Theorem \ref{yythm4.6}, the antipode 
$S$ is invertible, which will be used 
later in this proof. 

It is clear that a direct sum of finitely 
many AS Gorenstein algebras of dimension 
$d$ is also AS Gorenstein of dimension $d$, 
according to Definition~\ref{yydef3.1}.
Now considering $H \cong K_1 \oplus \dots 
\oplus K_n$, we can group together those 
$K_i$ of the same injective dimension, 
obtaining a decomposition 
$H \cong H_1 \oplus \dots \oplus H_m$ 
where each $H_i$ is an AS Gorenstein algebra 
of dimension $d_i$, say, and where 
$d_1 < d_2 < \dots < d_m$.  We now show that 
each $H_i$ must be a weak Hopf algebra and 
that the isomorphism is as weak Hopf algebras.  
Let $1 = e_1 + \dots + e_m$ be the decomposition 
into central pairwise commuting idempotents 
such that $H_i = e_i H$, and consider 
$H \lMod = H_1 \lMod \times \dots \times H_m \lMod$.

We verify the hypotheses of Lemma~\ref{yylem5.3}(2). 
Suppose that $V \in H \lMod$ is 
finite-dimensional. Note that if 
$V \in H_i\lMod$, then 
$\Ext^j_H(V, H) = \Ext^j_{H_i}(V, H_i) = 0$ 
for $j \neq d_i$.  If moreover $V \neq 0$, then 
$\Ext^{d_i}_H(V, H) = \Ext^{d_i}_{H_i}(V, H_i) \neq 0$ 
since $\Ext^{d_i}_{H_i}(-, H_i)$ is a duality 
from finite-dimensional modules in $H_i \lMod$ 
to finite-dimensional modules in $\rMod H_i$.  
Conversely, if $\Ext^j_H(V, H) = 0$ for all 
$j \neq d_i$ then the reverse argument shows 
that $V \in H_i \lMod$.  

Now let $V, W \in H \lMod$ both be finite-dimensional. 
Suppose that $V \in H_i \lMod$.  We have 
\[
\Ext^j_H(V \otensor^{\ell} W, H) 
\cong \Ext^j_H(V, H \otensor^{\ell} W^*) 
\cong \Ext^j_H(V, H \otensor^r (W^*)^S)  
\cong \Ext^j_H(V, H) \otensor^r (W^*)^S
\]
by Lemmas~\ref{yylem3.4} and \ref{yylem3.5}. 
Thus for $j \neq d_i$, since $\Ext^j_H(V,H) = 0$, 
we have $\Ext^j_H(V \otensor^{\ell} W, H) = 0$. 
Thus $V \otensor^{\ell} W \in H_i \lMod$ as well.

Similarly, suppose that $V, W \in H \lMod$ are 
finite-dimensional but now assume only that 
$W \in H_j \lMod$.  Since $S$ is invertible, 
$H^{\cop}$ is again a weak Hopf algebra, with 
the same algebra structure, opposite coproduct, 
and antipode $S^{-1}$.  Since $H^{\cop} = H$ as 
algebras, they have the same idempotent 
decomposition.  Applying the result of the 
previous paragraph to the weak Hopf algebra 
$H^{\cop}$ gives that $V \otensor^{\ell} W 
\in H_j \lMod$. 

Now if $V \in H_i \lMod$ and $W \in H_j \lMod$ 
are finite-dimensional, since $H_i \lMod 
\cap H_j \lMod = 0$ for $i \neq j$, 
$V \otensor^{\ell} W = 0$ in this case.  
If $i = j$ then we have $V \otensor^{\ell} W 
\in H_i \lMod$ as required.

Finally, consider a finite-dimensional 
$V \in H_i \lMod$ and let $V^*$ be its left 
dual in $H \lMod$.  As one of the axioms for 
the dual, the composition
\[
V^* \overset{1_{V^*} \otimes \coev}{\longrightarrow} 
V^* \otensor^{\ell} V \otensor^{\ell} V^* 
\overset{\ev \otimes 1_{V^*}}{\longrightarrow} V^*
\]
is equal to the identity on $V^*$.  By the 
previous paragraph, $V^* \otensor^{\ell} V 
\otensor^{\ell} V^* \in H_i \lMod$. Since 
this module surjects onto $V^*$, we have 
$V^* \in H_i \lMod$ as well.
\end{proof}

We refer to the paper \cite{RWZ1} for the undefined terms in the next corollary.

\begin{corollary}
\label{yycor5.5}
Let $H$ be an weak Hopf algebra which is 
finite over its affine center. Then $H$ is a 
direct sum of finitely many AS Gorenstein 
weak Hopf algebras. Each direct summand is 
Auslander Gorenstein, Cohen--Macaulay, and 
homogeneous of finite Gelfand–-Kirillov 
dimension equal to its injective dimension.
\end{corollary}

\begin{proof}
By \cite[Theorem 0.3]{RWZ1}, $H$ is a direct 
sum of finitely many AS Gorenstein algebras.  
The result now follows from Proposition 
\ref{yypro5.2}, Theorem \ref{yythm5.4} and 
\cite[Theorem 0.3]{RWZ1}.
\end{proof}

The corollary gives evidence that the following 
version of Brown--Goodearl question may have 
a positive answer.

\begin{question}\cite[Question 8.1]{RWZ1}
\label{yyque5.6}
Let $H$ be a noetherian weak Hopf algebra. 
Is $H$ isomorphic to a finite direct sum of 
AS Gorenstein weak Hopf algebras?
\end{question}

We close with the proof of the summary theorem 
from the introduction.

\begin{proof}[Proof of Theorem \ref{yythm0.6}]
Part (1) was already noted in Corollary~\ref{yycor5.5}.
Part (2) follows from part (1) and Theorem~\ref{yythm4.6}.  
Part (3) follows from Theorem \ref{yythm4.7}(2).
\end{proof}

\subsection*{Acknowledgments} 
R. Won was partially supported by an AMS--Simons 
Travel Grant and Simons Foundation grant \#961085. 
J.J. Zhang was partially supported by the US 
National Science Foundation (Nos. DMS-2001015 and 
DMS-2302087).  This material is based upon work 
supported by the National Science Foundation under 
Grant No.\ DMS-1928930 and by the Alfred P. Sloan 
Foundation under grant G-2021-16778, while the 
authors were in residence at the Simons Laufer 
Mathematical Sciences Institute (formerly MSRI) 
in Berkeley, California, during the Spring 2024 
semester.

\end{document}